\documentclass[a4paper,11pt]{article}
\usepackage[english]{babel}
\usepackage{tikz}
\usetikzlibrary{matrix,arrows,decorations.markings}
\usepackage{amsmath,amsfonts,amssymb,amsthm,url,textcomp}
\usepackage{csquotes}
\usepackage[left=2cm,right=2cm,bottom=3cm]{geometry}
\usepackage[numbers]{natbib}
\usepackage{fancyhdr}
\usepackage{anyfontsize}
\usepackage{hyperref}
\usepackage{hhline}
\usepackage{leftidx}

\newtheorem{theorem}{Theorem}\numberwithin{theorem}{section}

\newtheorem{lemma}[theorem]{Lemma}

\newtheorem{notation}[theorem]{Notation}
\newtheorem{theoremm}{Theorem}\numberwithin{theoremm}{subsection}
\newtheorem{deffinition}[theoremm]{Definition}
\newtheorem{lemmma}[theoremm]{Lemma}
\newtheorem{corrollary}[theoremm]{Corollary}
\newtheorem{propposition}[theoremm]{Proposition}
\newtheorem{theoremmm}{Theorem}\numberwithin{theoremmm}{subsubsection}
\newtheorem{lemmmma}[theoremmm]{Lemma}
\theoremstyle{remark}

\newcommand{\Rad}{\operatorname{Rad}}
\newcommand{\Aut}{\operatorname{Aut}}
\newcommand{\Alt}{\mathcal{A}}
\newcommand{\PSL}{\operatorname{PSL}}
\newcommand{\aff}{\mathrm{aff}}
\newcommand{\Aff}{\operatorname{Aff}}
\newcommand{\cl}{\operatorname{cl}}
\newcommand{\lcm}{\operatorname{lcm}}
\newcommand{\sh}{\operatorname{sh}}
\newcommand{\ord}{\operatorname{ord}}
\newcommand{\Sym}{\mathcal{S}}
\newcommand{\Hol}{\operatorname{Hol}}
\newcommand{\A}{\operatorname{A}}
\newcommand{\LL}{\operatorname{L}}
\newcommand{\C}{\operatorname{C}}
\newcommand{\el}{\operatorname{el}}
\newcommand{\Orb}{\operatorname{Orb}}
\newcommand{\meo}{\operatorname{meo}}
\newcommand{\mao}{\operatorname{mao}}
\newcommand{\CRRad}{\operatorname{CRRad}}
\newcommand{\Soc}{\operatorname{Soc}}
\newcommand{\Inn}{\operatorname{Inn}}
\newcommand{\id}{\operatorname{id}}
\newcommand{\N}{\operatorname{N}}
\newcommand{\fix}{\operatorname{fix}}
\newcommand{\Out}{\operatorname{Out}}
\newcommand{\e}{\mathrm{e}}
\newcommand{\T}{\mathcal{T}}

\newcommand{\PGL}{\operatorname{PGL}}
\newcommand{\Gal}{\operatorname{Gal}}
\newcommand{\GL}{\operatorname{GL}}
\newcommand{\Frob}{\operatorname{Frob}}
\newcommand{\PSU}{\operatorname{PSU}}
\newcommand{\PSp}{\operatorname{PSp}}
\newcommand{\PO}{\operatorname{P}\Omega}

\begin{document}

\title{Cycle lengths in finite groups and the size of the solvable radical}

\author{Alexander Bors\thanks{University of Salzburg, Mathematics Department, Hellbrunner Stra{\ss}e 34, 5020 Salzburg, Austria. \newline E-mail: \href{mailto:alexander.bors@sbg.ac.at}{alexander.bors@sbg.ac.at} \newline The author is supported by the Austrian Science Fund (FWF):
Project F5504-N26, which is a part of the Special Research Program \enquote{Quasi-Monte Carlo Methods: Theory and Applications}. \newline 2010 \emph{Mathematics Subject Classification}: primary: 20B25, 20D25, 20D45, secondary: 20D05, 20E22, 20G40, 37P99. \newline \emph{Key words and phrases:} finite groups, cycle structure, solvable radical, semisimple groups}}

\date{\today}

\maketitle

\abstract{We prove the following: For any $\rho\in\left(0,1\right)$, if a finite group $G$ has an automorphism with a cycle of length at least $\rho\cdot|G|$, then the index of the solvable radical $\Rad(G)$ in $G$ is bounded from above in terms of $\rho$, and such a condition is strong enough to imply solvability of $G$ if and only if $\rho>\frac{1}{10}$. Furthermore, considering, for exponents $e\in\left(0,1\right)$, the condition that a finite group $G$ have an automorphism with a cycle of length at least $|G|^e$, such a condition is strong enough to imply $|\Rad(G)|\to\infty$ for $|G|\to\infty$ if and only if $e>\frac{1}{3}$. We also prove similar results for a larger class of bijective self-transformations of finite groups, so-called periodic affine maps.}

\section{Introduction}\label{sec1}

\subsection{Motivation and main results}\label{subsec1P1}

In the author's preprint \cite{Bor15a}, we studied how having an automorphism with a \enquote{long} cycle restricts the structure of finite groups. One of the main results was that a finite group $G$ with an automorphism one of whose cycles has length greater than $\frac{1}{2}|G|$ is necessarily abelian. For proving these (and other) results, it turned out to be fruitful to study not only largest possible cycle lengths of automorphisms of finite groups $G$, but also of a more general type of bijective self-transformations, namely maps $\A_{g_0,\alpha}:G\rightarrow G$ of the form $g\mapsto g_0\alpha(g)$ for some fixed $g_0\in G$ and automorphism $\alpha$ of $G$. We called these maps \textit{periodic (left-)affine maps of $G$}.

Although most of the techniques introduced in \cite{Bor15a} work under weaker assumptions as well, one important idea (that an automorphism cycle, in a finite group $G$ with $|G|\geq 3$, of length at least $\frac{1}{2}|G|$ must intersect with its pointwise inverse) makes explicit use of the fraction $\frac{1}{2}$, and it is not clear how one could derive similar results under the assumption that $G$ have an automorphism cycle of length at least, say, $\frac{1}{3}|G|$.

We will tackle this problem here by studying consequences of conditions on finite groups $G$ of the form \enquote{$G$ has an automorphism cycle of length at least $\rho|G|$} (\enquote{first kind}) and \enquote{$G$ has a periodic affine map cycle of length at least $\rho|G|$} (\enquote{second kind}) for some fixed $\rho\in\left(0,1\right)$, and also of the form \enquote{$G$ has an automorphism cycle of length at least $|G|^e$} (\enquote{third kind}) and \enquote{$G$ has a periodic affine map cycle of length at least $|G|^e$} (\enquote{fourth kind}) for some fixed $e\in\left(0,1\right)$. Our main results, all of which rely on the classification of finite simple groups (CFSG), are as follows:

\begin{theoremm}\label{mainTheo1}
Let $\rho\in\left(0,1\right)$ be fixed, let $G$ be a finite group, and denote by $\Rad(G)$ the solvable radical of $G$. Then:

\noindent (1) If $G$ has an automorphism cycle of length at least $\rho|G|$, then $[G:\Rad(G)]\leq\rho^{E_1}$, where $E_1=-1.778151\ldots$.

\noindent (2) If $G$ has a periodic affine map cycle of length at least $\rho|G|$, then $[G:\Rad(G)]\leq\rho^{E_2}$, where $E_2=-5.906890\ldots$.
\end{theoremm}

So finite groups satisfying a condition of one of the first two forms are \enquote{not too far from being solvable}. An interesting question is for which values of $\rho$ such a condition actually implies solvability. Note that by \cite[Theorem 1.1.7]{Bor15a}, for conditions of the first form, this is the case whenever $\rho>\frac{1}{2}$. However, since solvability is a weaker condition than abelianity, one may hope to be able to do better, and actually, we will prove:

\begin{corrollary}\label{mainCor}
Let $G$ be a finite group.

\noindent (1) If $G$ has an automorphism cycle of length greater than $\frac{1}{10}|G|$, then $G$ is solvable. On the other hand, the alternating group $\Alt_5$ has an automorphism cycle of length $6=\frac{1}{10}|\A_5|$.

\noindent (2) If $G$ has a periodic affine map cycle of length greater than $\frac{1}{4}|G|$, then $G$ is solvable. On the other hand, $\Alt_5$ has a periodic affine map cycle of length $15=\frac{1}{4}|\Alt_5|$.
\end{corrollary}

As for the conditions of the third and fourth kind mentioned above, we cannot expect results as strong as Theorem \ref{mainTheo1} (see the discussion after Lemma \ref{fdsLcmLem}), but we have the following:

\begin{theoremm}\label{mainTheo2}
(1) Let $\epsilon>0$ be fixed. Then for every $\xi>0$, there exists a constant $K(\epsilon,\xi)$ such that for all finite groups $G$ having an automorphism cycle of length at least $|G|^{\frac{1}{3}+\epsilon}$, we have $[G:\Rad(G)]\leq\max(K(\epsilon,\xi),|G|^{1-\frac{3}{2}\epsilon+\xi})$. In particular, under a condition of the third kind with $e:=\frac{1}{3}+\epsilon$, for all $\xi>0$, we have $|G|^{\frac{3}{2}\epsilon-\xi}=o(|\Rad(G)|)$ for $|G|\to\infty$.

\noindent (2) Let $\epsilon>0$ be fixed. Then for every $\xi>0$, there exists a constant $K_{\aff}(\epsilon,\xi)$ such that for all finite groups $G$ having a periodic affine map cycle of length at least $|G|^{\frac{2}{3}+\epsilon}$, we have $[G:\Rad(G)]\leq\max(K_{\aff}(\epsilon,\xi),|G|^{1-3\epsilon+\xi})$. In particular, under a condition of the fourth kind with $e:=\frac{2}{3}+\epsilon$, for all $\xi>0$, we have $|G|^{3\epsilon-\xi}=o(|\Rad(G)|)$ for $|G|\to\infty$.

\noindent (3) There exists a sequence $(G_n)_{n\in\mathbb{N}}$ of finite groups $G_n$ such that $|\Rad(G_n)|=1$ for all $n\in\mathbb{N}$, $|G_n|\to\infty$ for $n\to\infty$, and for all $n\in\mathbb{N}$, $G_n$ has an automorphism cycle of length greater than $|G_n|^{\frac{1}{3}}$ and a periodic affine map cycle of length greater than $|G_n|^{\frac{2}{3}}$.
\end{theoremm}

We remark that we can and will give explicit definitions for $K(\epsilon,\xi)$, $K_{\aff}(\epsilon,\xi)$ and the sequence $(G_n)_{n\in\mathbb{N}}$, see the proof of Theorem \ref{mainTheo2} at the end of Section \ref{sec3}.

\subsection{Outline}\label{subsec1P2}

In Section \ref{sec2}, we present the technical tools needed for proving our main results, some of which were already introduced in \cite{Bor15a} and are therefore given without proof here. None of them make use of the CFSG. It turns out that using (part of) these tools, the proof of all of the main results can be reduced to the proof of one technical lemma, namely Lemma \ref{mainLem}, which we will call the \enquote{main lemma}. It is a statement about maximum cycle lengths of automorphisms and of periodic affine maps of finite nonabelian characteristically simple groups, and its proof will use the CFSG. Section \ref{sec3} shows how the main lemma implies all the main results, and Section \ref{sec4} consists of the proof of the main lemma.

\subsection{Notation and terminology}\label{subsec1P3}

For the readers' convenience, we explain those parts of our notation that may be nonstandard. We denote by $\mathbb{N}$ the set of natural numbers (von Neumann ordinals, including $0$), and by $\mathbb{N}^+$ the set of positive integers. The image of a set $M$ under a function $f$ is denoted by $f[M]$. The identity function on a set $M$ is denoted by $\id_M$, and the symmetric group on $M$ is denoted by $\Sym_M$, except when $M$ is a natural number $n$, in which case we set $\Sym_n:=\Sym_{\{1,\ldots,n\}}$. Similarly, for a natural number $n$, $\Alt_n$ is the alternating group on $\{1,\ldots,n\}$.

Let $G$ be a group. For an element $r\in G$, we denote by $\tau_r:G\rightarrow G,g\mapsto rgr^{-1}$ the inner automorphism of $G$ with respect to $r$. The centralizer and normalizer of a subset $X\subseteq G$ are denoted by $\C_G(X)$ and $\N_G(X)$ respectively. As in Theorems \ref{mainTheo1} and \ref{mainTheo2}, $\Rad(G)$ denotes the solvable radical of $G$. For linguistical simplicity, we will frequently use the following notation, see also \cite[Definitions 1.1.1, 2.1.1 and 2.1.2]{Bor15a} as well as \cite{GMPS15a}:

\begin{deffinition}\label{lambdaDef}
(1) Let $\psi$ be a permutation of a finite set $X$. We denote by $\Lambda(\psi)$ the maximum length of one of the disjoint cycles into which $\psi$ decomposes, and set $\lambda(\psi):=\frac{1}{|X|}\Lambda(\psi)$.

\noindent (2) For a finite group $G$, we set $\Lambda(G):=\max_{\alpha\in\Aut(G)}{\Lambda(\alpha)}$ and $\lambda(G):=\frac{1}{|G|}\Lambda(G)$.

\noindent (3) For a finite group $G$, the group of periodic left-affine maps of $G$ is denoted by $\Aff(G)$. We set $\Lambda_{\aff}(G):=\max_{A\in\Aff(G)}{\Lambda(A)}$ and $\lambda_{\aff}(G):=\frac{1}{|G|}{\Lambda_{\aff}(G)}$.

\noindent (4) For a finite group $G$, we denote by $\meo(G)$ the maximum element order of $G$ and set $\mao(G):=\meo(\Aut(G))$, the maximum automorphism order of $G$.
\end{deffinition}

We also use some notation and terminology from the theory of finite dynamical systems:

\begin{deffinition}\label{fdsDef}
(1) A \textbf{finite dynamical system}, abbreviated henceforth by \textbf{FDS}, is a finite set $X$ together with a map $f:X\rightarrow X$ (a so-called \textbf{self-transformation of $X$}). It is called \textbf{periodic} if and only if $f$ is bijective.

\noindent (2) If $(X_1,f_1),\ldots,(X_r,f_r)$ are FDSs, their \textbf{product} is defined as the FDS $(X_1\times\cdots\times X_r,f_1\times\cdots\times f_r)$, where $f_1\times\cdots\times f_r$ is the self-transformation of $X_1\times\cdots\times X_r$ mapping $(x_1,\ldots,x_r)\mapsto(f_1(x_1),\ldots,f_r(x_r))$.

\noindent (3) If $(X,\psi)$ is a periodic FDS and $x\in X$, we denote the length of the cycle of $x$ under $\psi$ by $\cl_{\psi}(x)$.
\end{deffinition}

Finally, in this paper, $\exp$ mostly denotes the exponent of a group, although in the definition of $\Psi$ in Subsection \ref{subsec2P4}, it denotes the natural exponential function. $\log$ always denotes the natural logarithm, and for $c>1$, the logarithm with base $c$ is denoted by $\log_c$.

\section{Some tools}\label{sec2}

\subsection{Lemmata concerning maximum cycle lengths}\label{subsec2P1}

Lemma \ref{fdsLcmLem} below was used in the proof of \cite[Lemma 2.1.6]{Bor15a}, of which Lemma \ref{productLem} is a part.

\begin{lemmma}\label{fdsLcmLem}
Let $(X_1,\psi_1),\ldots,(X_r,\psi_r)$ be periodic FDSs, and let $x=(x_1,\ldots,x_r)\in X_1\times\cdots\times X_r$. Then $\cl_{\psi_1\times\cdots\times\psi_r}(x)=\lcm(\cl_{\psi_1}(x_1),\ldots,\cl_{\psi_r}(x_r))$. In particular, $\Lambda(\psi_1\times\cdots\times\psi_r)\leq\Lambda(\psi_1)\cdots\Lambda(\psi_r)$.\qed
\end{lemmma}

We remark that by Lemma \ref{fdsLcmLem}, any condition on finite groups $G$ of the form $\lambda(G)\geq f(|G|)$, where $f:\mathbb{N}^+\rightarrow\left[0,1\right]$ is such that $f(n)\to 0$ for $n\to\infty$, is not strong enough to imply that the index $[G:\Rad(G)]$ is bounded from above. Indeed, under such a condition, any finite group $G_0$ (in particular, any nonabelian finite simple group $G_0$) may occur as a direct factor of $G$. To see this, let $p$ be a prime which is so large that $f(p|G_0|)\leq\frac{1}{2|G_0|}$. Considering the product automorphism $\id_{G_0}\times\alpha$ of $G:=G_0\times\mathbb{Z}/p\mathbb{Z}$, where $\alpha\in\Aut(\mathbb{Z}/p\mathbb{Z})$ is the multiplication by any primitive root modulo $p$, it is not difficult to see by Lemma \ref{fdsLcmLem} that \[\lambda(G)\geq\lambda(\id_{G_0}\times\alpha)=\frac{p-1}{p|G_0|}=(1-\frac{1}{p})\cdot\frac{1}{|G_0|}\geq\frac{1}{2|G_0|}\geq f(|G|).\]

As in \cite{Bor15a}, we say that a family $(G_i)_{i\in I}$ of groups has the \textit{splitting property} if and only if for every automorphism $\alpha$ of $\prod_{i\in I}{G_i}$, there exists a family $(\alpha_i)_{i\in I}$ such that $\alpha_i$ is an automorphism of $G_i$ for $i\in I$, and $\alpha((g_i)_{i\in I})=(\alpha_i(g_i))_{i\in I}$ for all $(g_i)_{i\in I}\in\prod_{i\in I}{G_i}$.

\begin{lemmma}\label{productLem}
Let $(G_1,\ldots,G_r)$ be a tuple of finite groups with the splitting property. Then:

\noindent (1) $\Lambda(G_1\times\cdots\times G_r)\leq\Lambda(G_1)\cdots\Lambda(G_r)$.

\noindent (2) For every periodic affine map $A$ of $G_1\times\cdots\times G_r$, there exists a tuple $(A_1,\ldots,A_r)$ such that $A_i\in\Aff(G_i)$ for $i=1,\ldots,r$ and $A=A_1\times\cdots\times A_r$. In particular, $\Lambda_{\aff}(G_1\times\cdots\times G_r)\leq\Lambda_{\aff}(G_1)\cdots\Lambda_{\aff}(G_r)$.\qed
\end{lemmma}

The following is a part of \cite[Lemma 2.1.4]{Bor15a}:

\begin{lemmma}\label{transferLem}
Let $G$ be a finite group, $N$ a characteristic subgroup of $G$. Then:

\noindent (1) $\Lambda(G)\leq\Lambda_{\aff}(N)\cdot\Lambda(G/N)$, or equivalently, $\lambda(G)\leq\lambda_{\aff}(N)\cdot\lambda(G/N)$. In particular, $\lambda(G/N)\geq\lambda(G)$.

\noindent (2) $\Lambda_{\aff}(G)\leq\Lambda_{\aff}(N)\cdot\Lambda_{\aff}(G/N)$, or equivalently, $\lambda_{\aff}(G)\leq\lambda_{\aff}(N)\cdot\lambda_{\aff}(G/N)$. In particular, $\lambda_{\aff}(G)\leq\min(\lambda_{\aff}(N),\lambda_{\aff}(G/N))$.\qed
\end{lemmma}

We will now prove some more results that are useful for the study of $\Lambda_{\aff}$-values of finite groups. For a more concise formulation, we define:

\begin{deffinition}\label{shiftDef}
Let $G$ be a finite group, $x\in G$, $\alpha$ an automorphism of $G$, $n\in\mathbb{N}^+$.

\noindent (1) The element $\sh_{\alpha}^{(n)}(x):=x\alpha(x)\cdots\alpha^{n-1}(x)\in G$ is called the \textbf{$n$-th shift of $x$ under $\alpha$}.

\noindent (2) The element $\sh_{\alpha}(x):=\sh_{\alpha}^{(\ord(\alpha))}\in G$ is called the \textbf{shift of $x$ under $\alpha$}.
\end{deffinition}

The following calculation rules for shifts are easy to show:

\begin{lemmma}\label{shiftLem}
Let $G$ be a finite group, $x\in G$, $\alpha$ an automorphism of $G$.

\noindent (1) $\alpha(\sh_{\alpha}(x))=x\sh_{\alpha}(x)x^{-1}$.

\noindent (2) If $d\in\mathbb{N}^+$ is such that $\cl_{\alpha}(x)\mid d\mid \ord(\alpha)$, then $\sh_{\alpha}(x)=\sh_{\alpha}^{(d)}(x)^{\frac{\ord{\alpha}}{d}}$.\qed
\end{lemmma}

Definition \ref{shiftDef} is motivated by the following: It is well-known that there is natural isomorphism between $\Aff(G)$, the product, inside $\Sym_G$, of the image of the left regular representation of $G$ with $\Aut(G)$, and the holomorph of $G$, $\Hol(G)=G\rtimes\Aut(G)$. The isomorphism is simply given by the map $\Aff(G)\rightarrow\Hol(G),\A_{x,\alpha}\mapsto(x,\alpha)$. It is therefore clear that $\ord(\alpha)\mid\ord(\A_{x,\alpha})$ for all $x\in G$ and all $\alpha\in\Aut(G)$, and thus $\ord(\A_{x,\alpha})=\ord(\alpha)\cdot\ord(\A_{x,\alpha}^{\ord(\alpha)})$. However, easy computations reveal that under said natural isomorphism, $\A_{x,\alpha}^{\ord(\alpha)}$ corresponds to the element $\sh_{\alpha}(x)\in G$. This shows that in general, we have the following formula for computing orders of periodic affine maps of finite groups: \[\ord(\A_{x,\alpha})=\ord(\alpha)\cdot\ord(\sh_{\alpha}(x)).\]

When $\psi$ is a permutation of a finite set $X$ and $n\in\mathbb{N}^+$, we say that an orbit $O$ of the action of $\psi$ on $X$ \textit{induces} an orbit $\tilde{O}$ of $\psi^n$ (or that $\tilde{O}$ \textit{stems from} $O$) if and only if $\tilde{O}\subseteq O$, in which case $|\tilde{O}|=\frac{1}{\gcd(n,|O|)}|O|$. Every orbit of $\psi$ induces an orbit of $\psi^n$, and every orbit of $\psi^n$ stems from precisely one orbit of $\psi$.

\begin{lemmma}\label{divisorLem}
Let $G$ be a finite group, $x\in G$, $\alpha$ an automorphism of $G$. Then every cycle length of $\A_{x,\alpha}$ is divisible by $\LL_G(x,\alpha):=\ord(\sh_{\alpha}(x))\cdot\prod_{p}{p^{\nu_p(\ord(\alpha))}}$, where $p$ runs through the common prime divisors of $\ord(\sh_{\alpha}(x))$ and $\ord(\alpha)$. In particular, $\LL_G(x,\alpha)\mid|G|$.
\end{lemmma}

\begin{proof}
Every orbit of $\A_{x,\alpha}^{\ord(\alpha)}$, the left multiplication by $\sh_{\alpha}(x)$ in $G$, has size $\ord(\sh_{\alpha}(x))$, so certainly every cycle length of $\A_{x,\alpha}$ is divisible by $\ord(\sh_{\alpha}(x))$. In particular, if $p$ is a common prime divisor of $\ord(\sh_{\alpha}(x))$ and $\ord(\alpha)$, and $O$ is any orbit of $\A_{x,\alpha}$, then $p\mid|O|$, but $p^{\nu_p(\ord(\sh_{\alpha}(x)))}$ still divides $|\tilde{O}|$, where $\tilde{O}$ is the orbit of $\A_{x,\alpha}^{\ord(\alpha)}$ induced by $O$. This is only possible if $|O|$ actually is divisible by $p^{\nu_p(\ord(\sh_{\alpha}(x)))+\nu_p(\ord(\alpha))}$, and the assertion follows.
\end{proof}

\begin{lemmma}\label{centralizerLem}
Let $G$ be a finite group, $x,r\in G$. Then $x^{-1}r\in\C_G(\sh_{\tau_r}(x))$. In particular, if, for some subgroup $H\leq G$, $\C_G(\sh_{\tau_r}(x))\subseteq H$, then $x\in H$ if and only if $r\in H$.
\end{lemmma}

\begin{proof}
This follows immediately from $r\sh_{\tau_r}(x)r^{-1}=\tau_r(\sh_{\tau_r}(x))=x\sh_{\tau_r}(x)x^{-1}$, where the first equality is by the definition of $\tau_r$ and the second by Lemma \ref{shiftLem}(1).
\end{proof}

\begin{lemmma}\label{lcmLem}
(1) Let $G$ be a finite centerless group, $r,s\in G$. Set $x:=sr^{-1}$. Then $\sh_{\tau_r}(x)=s^{\ord(r)}$. In particular, $\ord(\A_{x,\tau_r})=\lcm(\ord(s),\ord(r))$.

\noindent (2) Let $G$ be any finite group, $r,s\in G$, $x$ as in point (1). Then $\sh_{\tau_r}(x)=s^{\ord(\tau_r)}\cdot r^{-\ord(\tau_r)}$. In particular, if $\gcd(\ord(r),\ord(s))=1$, then $\ord(\A_{x,\tau_r})=\ord(s)\cdot\ord(r)$.
\end{lemmma}

\begin{proof}
An easy induction on $n\in\mathbb{N}^+$ proves that in both cases, we have $\sh_{\tau_r}^{(n)}(x)=s^nr^{-n}$. Therefore, we have $\sh_{\tau_r}(x)=s^{\ord(r)}$ under the assumptions of point (1). This implies that \[\ord(\A_{x,\tau_r})=\ord(\tau_r)\cdot\ord(\sh_{\tau_r}(x))=\ord(r)\cdot\frac{\ord(s)}{\gcd(\ord(s),\ord(r))}=\lcm(\ord(s),\ord(r)),\] proving the statement of point (1). The proof of point (2) is similar, using that $r^{-\ord(\tau_r)}\in\zeta G$ and that the order of a product of two commuting elements with coprime orders is the product of their orders.
\end{proof}

\subsection{Some results on finite semisimple groups}\label{subsec2P2}

In this Subsection, for the readers' convenience, we first briefly recall some basic facts on finite semisimple groups (finite groups without nontrivial solvable normal subgroups) which we will need later, following mostly the exposition in \cite[pp.~89ff.]{Rob96a}. Afterward, we generalize a result of Horo\v{s}evski\u{\i} on largest cycle lengths of automorphisms of finite semisimple groups to periodic affine maps of such groups.

Any group $G$ has a unique largest normal centerless CR-subgroup, the centerless CR-radical of $G$, which we denote by $\CRRad(G)$. From now on, assume that $G$ is finite and semisimple. Then $\CRRad(G)$ coincides with $\Soc(G)$, the socle of $G$. $G$ canonically embeds into $\Aut(\Soc(G))$ by its conjugation action (which shows that for any finite centerless CR-group $R$, there are only finitely many isomorphism types of finite semisimple groups $G$ such that $\Soc(G)\cong R$), and the image $G^{\ast}$ of this embedding clearly contains $\Inn(\Soc(G))$. Conversely, for every finite centerless CR-group $R$, any group $G$ such that $\Inn(R)\leq G\leq\Aut(R)$ is semisimple.

If $S_1,\ldots,S_r$ are pairwise nonisomorphic nonabelian finite simple groups, and $n_1,\ldots,n_r\in\mathbb{N}^+$, then the tuple $(S_1^{n_1},\ldots,S_r^{n_r})$ has the splitting property. In particular, $\Aut(S_1^{n_1}\times\cdots\times S_r^{n_r})=\Aut(S_1^{n_1})\times\cdots\times\Aut(S_r^{n_r})$. The structure of the automorphism groups of finite nonabelian characteristically simple groups (powers of finite nonabelian simple groups) can be described by permutational wreath products. More precisely, $\Aut(S^n)=\Aut(S)\wr\Sym_n$ for any finite nonabelian simple group $S$ and any $n\in\mathbb{N}^+$.

Rose \cite[Lemma 1.1]{Ros75a} observed that, in generalization of the embedding of $G$ into $\Aut(\Soc(G))$ for finite semisimple groups $G$, if $G$ is any group, and $H$ a characteristic subgroup of $G$ such that $\C_G(H)=\{1_G\}$, then $G$ embeds into $\Aut(H)$ by its conjugation action on $H$, and, viewing $G$ as a subgroup of $\Aut(H)$, $\Aut(G)$ is canonically isomorphic to $\N_{\Aut(H)}(G)$. This implies, among other things, that automorphism groups of finite centerless CR-groups are complete.

Let us now turn to the aforementioned theorem of Horo\v{s}evski\u{\i}. Following the terminology from \cite{GPS15a}, we define:

\begin{deffinition}\label{regularDef}
Let $\psi$ be a permutation of a finite set. A cycle of $\psi$ whose length equals $\ord(\psi)$ is called a \textbf{regular cycle of $\psi$}.
\end{deffinition}

Thus a permutation $\psi$ of a finite set has a regular cycle if and only if $\Lambda(\psi)=\ord(\psi)$. In the case of periodic affine maps $A$ of finite groups $G$, the order is often easier to compute than the $\Lambda$-value, since for computing the order, one can work with the compact representation $A=\A_{x,\alpha}$ for appropriate $x\in G$ and $\alpha\in\Aut(G)$, and composition of periodic affine maps translates, on the level of the compact representations, into some simple manipulations (by the isomorphism $\Aff(G)\rightarrow\Hol(G)$ mentioned above), without the need to \enquote{spread out} the entire element structure of $G$ to determine the cycle lengths of the elements of $G$ under $A$.

In view of this, it would be nice to know at least for some classes of finite groups $G$ that all periodic affine maps of $G$ have a regular cycle to make computation of $\Lambda$- and $\Lambda_{\aff}$-values easier. Indeed, Horo\v{s}evski\u{\i} proved:

\begin{theoremm}\label{horTheo}(\cite[Theorem 1]{Hor74a})
Let $G$ be a finite semisimple group. Then every automorphism of $G$ has a regular cycle.\qed
\end{theoremm}

We will extend this to:

\begin{theoremm}\label{regularTheo}
Let $G$ be a finite semisimple group. Then every periodic affine map of $G$ has a regular cycle.
\end{theoremm}

Our proof of Theorem \ref{regularTheo} is mostly an adaptation of Horo\v{s}evski\u{\i}'s proof of Theorem \ref{horTheo}, with the arguments getting slightly more complicated because of the more general situation. However, at one point, our proof significantly differs from the one of Horo\v{s}evski\u{\i}, using the recent result \cite[Theorem 3.2]{GPS15a} to settle one important case. Just like Horo\v{s}evski\u{\i}, we use the following:

\begin{lemmma}\label{regularLem}
Let $X$ be a finite set, $\psi\in\Sym_X$, $p$ a prime such that $p^2\mid\ord(\psi)$. The following are equivalent:

\noindent (1) $\psi$ has a regular cycle.

\noindent (2) $\psi^p$ has a regular cycle.
\end{lemmma}

\begin{proof}
See \cite[Lemma 1]{Hor74a}. The assumption there that $\psi$ (called $\phi$ there) is an automorphism of a finite group is not needed.
\end{proof}

Before we continue with the next lemma, a quick reminder and an easy observation: Recall that for a group $G$, an automorphism $\alpha$ of $G$, and a normal subgroup $N\unlhd G$, $\alpha$-admissibility of $N$ (i.e., the property that $\alpha(N)=N$) is equivalent to the existence of an automorphism $\tilde{\alpha}$ of $G/N$ such that, denoting by $\pi:G\rightarrow G/N$ the canonical projection, $\pi\circ\alpha=\tilde{\alpha}\circ\pi$. In this case, $\tilde{\alpha}$ is unique and is called the \textit{automorphism of $G/N$ induced by $\alpha$}. More generally, if, for some permutation $\psi$ of $G$, there exists a permutation $\sigma$ of $G/N$ such that $\pi\circ\psi=\sigma\circ\pi$, we still call $\sigma$ \textit{induced by $\psi$}. It is not difficult to see that for any group $G$, any $N\unlhd G$ and any periodic affine map $A=\A_{x,\alpha}$ of $G$, $A$ induces a permutation $\tilde{A}$ of $G/N$ if and only if $N$ is $\alpha$-admissible, and in this case, $\tilde{A}$ is a periodic affine map of $G/N$; actually, $\tilde{A}=\A_{\pi(x),\tilde{\alpha}}$.

\begin{lemmma}\label{centQuotLem}
Let $G$ be a group, $B\unlhd G$, $A$ a periodic affine map of $G$ such that $A_{\mid B}=\id_B$. Then $\C_G(B)\unlhd G$, and $A$ induces the identity map in $G/\C_G(B)$.
\end{lemmma}

\begin{proof}
In general, for all $x\in G$ and $\alpha\in\Aut(G)$, it follows immediately from the definition of $\A_{x,\alpha}$ that $\A_{x,\alpha}(1_G)=x$. Since $A(1_G)=1_G$ by assumption, $A$ thus actually is an automorphism of $G$, so the claim follows from \cite[Lemma 2]{Hor74a}.
\end{proof}

\begin{lemmma}\label{productRegLem}
Let $X_1,\ldots,X_n$ be finite sets, $\psi_i$, $i=1,\ldots,n$, a permutation of $X_i$ with a regular cycle. Then $\psi_1\times\cdots\times\psi_n$ has a regular cycle.\qed
\end{lemmma}

One additional easy observation which we will need is the following:

\begin{lemmma}\label{fixLem}
Let $G$ be a group, $A=\A_{x,\alpha}$ a periodic left affine map of $G$ such that $\fix(A)\not=\emptyset$. Then $\fix(A)$ is a left coset of the subgroup $\fix(\alpha)\leq G$.
\end{lemmma}

\begin{proof}
For all $g\in G$, we have that $g\in\fix(A)$ if and only if $x\alpha(g)=g$, or $x=g\alpha(g)^{-1}$. Therefore, if we fix $f\in\fix(A)$, then $\fix(A)$ can be desribed as $\{g\in G\mid g\alpha(g)^{-1}=f\alpha(f)^{-1}\}=\{g\in G\mid g^{-1}f\in\fix(\alpha)\}=f\fix(\alpha)$.
\end{proof}

\begin{proof}[Proof of Theorem \ref{regularTheo}]
The proof is by induction on $|G|$, the induction base $|G|=1$ being trivial, with an inner induction on $\ord(A)$, the induction base $\ord(A)=1$ being trivial. For the induction step, assume that $A=\A_{x,\alpha}$ is a periodic affine map of the finite semisimple group $G$. To show that $A$ has a regular cycle, we make a case distinction:

\begin{enumerate}
\item Case: $G$ is simple. This case is by contradiction, so assume that $A$ does not have a regular cycle. Note that by Lemma \ref{regularLem} and the induction hypothesis, $\ord(A)$ then must be squarefree, say $\ord(A)=p_1\cdots p_r$, with the $p_i$ pairwise distinct primes. Since by the induction hypothesis, $A^{p_1}$ has a cycle of length $\ord(A^{p_1})=p_2\cdots p_r$, but $A$ has no regular cycle, $A$ must also have a cycle of length $p_2\cdots p_r$, which implies $p_2\cdots p_r<|G|$. Now note that by the assumption that $A$ does not have a regular cycle, we have $G\subseteq\bigcup_{i=1}^r{\fix(A^{\prod_{j\not=i}{p_j}})}$. By Lemma \ref{fixLem}, denoting by $\alpha_i$ the underlying automorphism of $A^{\prod_{j\not=i}{p_j}}$, we have $|\fix(A^{\prod_{j\not=i}{p_j}})|=|\fix(\alpha_i)|$, and so there must exist $i\in\{1,\ldots,r\}$ such that $[G:\fix(\alpha_i)]\leq r$ (otherwise, $G$ could not be covered by the $r$ fixed point sets above). But since $G$ is simple, this implies that $|G|\leq r!\leq p_2\cdots p_r<|G|$, a contradiction.

\item Case: $G$ is characteristically simple, but not simple. Let $S$ be a nonabelian finite simple group and $n\geq 2$ such that $G\cong S^n$. $\alpha$ is an element of the permutational wreath product $\Aut(S)\wr\Sym_n$, i.e., $\alpha$ is a composition $(\alpha_1\times\cdots\times\alpha_n)\circ\psi$, where each $\alpha_i$ is an automorphism of $S$ and $\psi$ is a permutation of coordinates on $S^n$. Writing $x=(x_1,\ldots,x_n)$, and denoting by $\mu_{x}$ the left multiplication by $x$ in $S^n$, it follows that $A=\mu_{x}\circ((\alpha_1\times\cdots\times\alpha_n)\circ\psi)=((\mu_{x_1}\times\cdots\times\mu_{x_n})\circ(\alpha_1\times\cdots\times\alpha_n))\circ\psi=(\A_{x_1,\alpha_1}\times\cdots\times\A_{x_n,\alpha_n})\circ\psi$. This proves that $A\in\Aff(S)\wr\Sym_n$ (actually, we just proved that $\Aff(S^n)=\Aff(S)\wr\Sym_n$). By induction hypothesis, every permutation from $\Aff(S)$ has a regular cycle, and so by \cite[Theorem 3.2]{GPS15a}, $A$ has a regular cycle.

\item Case: $G$ is completely reducible, but not characteristically simple. Let $S_1,\ldots,S_r$ be pairwise nonisomorphic nonabelian finite simple groups, $n_1,\ldots,n_r\in\mathbb{N}^+$ such that $G\cong S_1^{n_1}\times\cdots\times S_r^{n_r}$, and note that $r\geq 2$ by assumption. Since $(S_1^{n_1},\ldots,S_r^{n_r})$ has the splitting property, by Lemma \ref{productLem}(2), $A$ can be written as a product of periodic affine maps over the single $S_i^{n_i}$, each of which has a regular cycle by the induction hypothesis, and so $A$ has a regular cycle by Lemma \ref{productRegLem}.

\item Case: $G$ is not completely reducible. Set $B:=\Soc(G)$, and note that $B$ is proper in $G$ and $\C_G(B)=\{1_G\}$. Denote by $\tilde{A}$ the periodic affine map of $G/B$ induced by $A$, and let $k$ denote the length of the identity element of $G/B$ under $\tilde{A}$. Set $A_0:=A^k$. Then $A_0$ restricts to a periodic affine map of $B$, so by the induction hypothesis, ${A_0}_{\mid B}$ has a cycle of length $n:=\ord({A_0}_{\mid B})$; fix an element $x\in B$ such that $\cl_{A_0}(x)=n$. Now $A_0^n$ acts identically in $B$, and thus by Lemma \ref{centQuotLem} also in $G\cong G/\C_G(B)$. This means that $n=\ord(A_0)$, and so $\ord(A)\leq k\cdot n$. But clearly, $\cl_{A}(x)=k\cdot n$, since $k$ divides the cycle length under $A$ of any element from $B$. Therefore, $\ord(A)=k\cdot n$ and $A$ has a regular cycle.
\end{enumerate}
\end{proof}

\begin{corrollary}\label{regularCor}
(1) Let $G$ be a finite semisimple group. Then:

(i) $\Lambda(G)=\mao(G)$.

(ii) $\Lambda_{\aff}(G)=\meo(\Hol(G))$.

\noindent (2) Let $R$ be a finite centerless CR-group. Then:

(i) $\Lambda(\Aut(R))=\mao(R)$.

(ii) $\Lambda_{\aff}(\Aut(R))=\meo(\Hol(\Aut(R)))$.
\end{corrollary}

\begin{proof}
For (1): (i) is an immediate consequence of Theorem \ref{regularTheo}, and (ii) also follows from Theorem \ref{regularTheo} and the fact that $\Aff(G)\cong\Hol(G)$.

For (2): As for (i), note that $\Aut(R)$ is semisimple, and so by (1,i), we have $\Lambda(\Aut(R))=\mao(\Aut(R))=\meo(\Aut(\Aut(R)))=\meo(\Aut(R))=\mao(R)$, where the second-to-last equality follows from the completeness of $\Aut(R)$. (ii) just is a special case of (1,ii).
\end{proof}

\subsection{Upper bounds on element orders in wreath products}\label{subsec2P3}

We will need upper bounds on $\meo(G)$ and $\mao(G)$ for finite semisimple groups $G$. To this end, some bounds on orders of elements in wreath products in general come in handy. Before formulating and proving Lemma \ref{wreathLem} below, we introduce the following notation and terminology:

\begin{deffinition}\label{wreathDef}
Let $G$ be a finite group, $n\in\mathbb{N}^+$, and $\psi\in\Sym_n$.

\noindent (1) Let $g=(g_1,\ldots,g_n)\in G^n$. For $i=1,\ldots,n$, we define $\el_i^{(\psi)}(g):=g_ig_{\psi^{-1}(i)}\cdots g_{\psi^{-\cl_{\psi}(i)+1}(i)}\in G$. Alternatively, one can describe $\el_i^{(\psi)}(g)$ as the image of $\sh_{\tau_{\psi}}^{(\cl_{\psi}(i))}(g)\in G^n\leq G\wr\Sym_n$ under the projection $\pi_i:G^n\rightarrow G$ onto the $i$-th component.

\noindent (2) We denote the set of orbits of the action of $\psi$ on $\{1,\ldots,n\}$ by $\Orb(\psi)$.

\noindent (3) An \textbf{assignment to $\psi$ in $G$} is a function $\beta:\Orb(\psi)\rightarrow G$. For such an assignment $\beta$, we define its \textbf{order} to be the least common multiple of the numbers $\ord(\beta(O)^{\frac{\ord(\beta)}{|O|}})$, where $O$ runs through $\Orb(\psi)$.
\end{deffinition}

\begin{lemmma}\label{wreathLem}
Let $G$ be a finite group, $n\in\mathbb{N}^+$, denote by $\pi:G\wr\Sym_n\rightarrow\Sym_n$ the canonical projection, and let $\psi\in\Sym_n$.

\noindent (1) Let $g=(g_1,\ldots,g_n)\in G^n$ and consider the element $x:=(g,\psi)\in G^n\rtimes\Sym_n=G\wr\Sym_n$. Then for $i=1,\ldots,n$, the $i$-th component of $x^{\ord(\psi)}\in G^n$ equals $\el_i^{(\psi)}(g)^{\frac{\ord(\psi)}{\cl_{\psi}(i)}}$.

\noindent (2) In particular, the maximum order of an element $x\in G\wr\Sym_n$ such that $\pi(x)=\psi$ equals the product of $\ord(\psi)$ with the maximum order of an assignment to $\psi$ in $G$ and is bounded from above by $\ord(\psi)\cdot\meo(G^{|\Orb(\psi)|})$.
\end{lemmma}

\begin{proof}
For (1): We may assume that $G$ is nontrivial. Fix $i$, and denote by $\pi_i:G^n\rightarrow G$ the projection onto the $i$-th component. It is clear that $x^{\ord(\psi)}=\sh_{\tau_{\psi}}(g)$ (where the shift is formed inside $G\wr\Sym_n$ and $\tau_{\psi}$ is the inner automorphism of $G\wr\Sym_n$ with respect to $\psi$), whence $\pi_i(x^{\ord(\psi)})=\pi_i(\sh_{\tau_{\psi}}(g))$. But the $i$-th component of $\sh_{\tau_{\psi}}(g)$ only depends on the components of $g$ whose indices are from the orbit $O_i$ of $i$ under $\psi$, so if we denote by $\tilde{g}$ the element of $G^n$ which has the same entries as $g$ in the components whose indices are in $O_i$ but all other entries equal to $1_G$, we have $\pi_i(x^{\ord(\psi)})=\pi_i(\sh_{\tau_{\psi}}(\tilde{g}))$. Now note that $\cl_{\psi}(i)$ is a multiple of $\cl_{\tau_{\psi}}(\tilde{g})$ and a divisor of $\ord(\psi)=\ord(\tau_{\psi})$, which gives us, by an application of Lemma \ref{shiftLem}(2), \[\pi_i(x^{\ord(\psi)})=\pi_i(\sh_{\tau_{\psi}}(\tilde{g}))=\pi_i(\sh_{\tau_{\psi}}^{(\cl_{\psi}(i))}(\tilde{g})^{\frac{\ord(\psi)}{\cl_{\psi}(i)}})=\pi_i(\sh_{\tau_{\psi}}^{(\cl_{\psi}(i))}(\tilde{g}))^{\frac{\ord(\psi)}{\cl_{\psi}(i)}}=\]\[\pi_i(\sh_{\tau_{\psi}}^{(\cl_{\psi}(i))}(g))^{\frac{\ord(\psi)}{\cl_{\psi}(i)}}=\el_i^{(\psi)}(g)^{\frac{\ord(\psi)}{\cl_{\psi}(i)}}.\]

\noindent For (2): For any element $x\in G\wr\Sym_n$ of the form $(g,\psi)$, we have $\ord(x)=\ord(\psi)\cdot\ord(x^{\ord(\psi)})$, where, by (1), the second factor is the least common multiple of the numbers $\ord(\el_i^{(\psi)}(g)^{\frac{\ord(\psi)}{\cl_{\psi}(i)}})$ for $i=1,\ldots,n$. Fix a set $\mathcal{R}$ of representatives of the orbits of $\psi$, which is in canonical bijection with $\Orb(\psi)$. It is not difficult to see that if $i,j\in\{1,\ldots,n\}$ are from the same orbit under $\psi$, then $\el_i^{(\psi)}(g)^{\frac{\ord(\psi)}{\cl_{\psi}(i)}}$ and $\el_j^{(\psi)}(g)^{\frac{\ord(\psi)}{\cl_{\psi}(j)}}$ are conjugate in $G$ and thus have the same order, so $\ord(x^{\ord(\psi)})$ is equal to just the least common multiple of the numbers $\ord(\el_i^{(\psi)}(g)^{\frac{\ord(\psi)}{\cl{\psi}(i)}})$ for $i\in\mathcal{R}$. Therefore, composing the canonical bijection $\Orb(\psi)\to\mathcal{R}$ with the function $\mathcal{R}\rightarrow G,i\mapsto\el_i^{(\psi)}(g)$ gives an assignment to $\psi$ in $G$ whose order coincides with $\ord(x^{\ord(\psi)})$. Conversely, if any assignment $\beta$ to $\psi$ in $G$ is given, by choosing the components $g_1,\ldots,g_n$ of $G$ such that for all $O\in\Orb(\psi)$ there exists $i\in O$ such that $g_ig_{\psi^{-1}(i)}\cdots g_{\psi^{-\cl_{\psi}(i)+1}(i)}=\beta(O)$, we can assure that $\ord((g,\psi)^{\ord(\psi)})=\ord(\beta)$. This proves the claim.
\end{proof}

\subsection{Landau's and Chebyshev's function}\label{subsec2P4}

Both Landau's function $g:\mathbb{N}^+\rightarrow\mathbb{N}^+,n\mapsto\meo(\Sym_n)$, and Chebyshev's function $\psi:\mathbb{N}^+\rightarrow\mathbb{N}^+,n\mapsto\log(\exp(\Sym_n))$, are well-studied in analytic number theory. Apart from information on their asymptotic growth behavior, explicit upper bounds are also available. More explicitly, Massias \cite[Th{\'e}or{\`e}me, p.~271]{Mas84a} proved that $\log(g(n))\leq 1.05314\cdot\sqrt{n\log(n)}$ for all $n\in\mathbb{N}^+$, and Rosser and Schoenfeld \cite[Theorem 12]{RS62a} that $\psi(n)<1.03883\cdot n$ for all $n\in\mathbb{N}^+$.

The latter result translates into an exponential upper bound on $\Psi:=\exp\circ\psi$. For $n\leq 27$, the following best possible exponential bound on $g(n)$ is sharper than the subexponential bound by Massias, and its use will make some of our arguments easier:

\begin{propposition}\label{landauProp}
For all $n\in\mathbb{N}^+$, we have $g(n)\leq 3^{\frac{n}{3}}$, with equality if and only if $n=3$.\qed
\end{propposition}

We conclude with the following consequence of Lemma \ref{wreathLem}:

\begin{lemmma}\label{boundLem}
(1) Let $G$ be a finite group, $n\in\mathbb{N}+$. Then $\meo(G\wr\Sym_n)\leq g(n)\cdot\meo(G^n)$.

\noindent (2) Let $S$ be a nonabelian finite simple group, $n\in\mathbb{N}$. Then $g(n)\cdot\meo(\Aut(S)^n)<|S|^{n/3}$ implies that $\Lambda(\Aut(S^n))<|S^n|^{1/3}$ and $\Lambda_{\aff}(\Aut(S^n))<|S^n|^{2/3}$.
\end{lemmma}

\begin{proof}
For (1): This follows immediately from Lemma \ref{wreathLem}(2).

\noindent For (2): Using Corollary \ref{regularCor}(2), we conclude that $\Lambda(\Aut(S^n))=\meo(\Aut(S^n))=\meo(\Aut(S)\wr\Sym_n)\leq g(n)\cdot\meo(\Aut(S)^n)<|S|^{n/3}=|S^n|^{1/3}$, and that $\Lambda_{\aff}(\Aut(S^n))=\meo(\Hol(\Aut(S^n)))=\meo(\Aut(S^n)\rtimes\Aut(\Aut(S^n)))\leq\meo(\Aut(S^n))\cdot\meo(\Aut(\Aut(S^n)))=\meo(\Aut(S^n))^2<|S^n|^{2/3}$.
\end{proof}

\section{Reduction to the main lemma}\label{sec3}

The aforementioned \enquote{main lemma} is the following:

\begin{lemma}\label{mainLem}
Let $G$ be a finite nonabelian characteristically simple group. Then:

\noindent (1) $\Lambda(\Aut(G))<|G|^{\frac{1}{3}}$, with the following exceptions:

(i) $G\cong\PSL_2(q)$ for some primary $q\geq 5$. In this case, $\Lambda(\Aut(G))=q+1$, we have $\frac{1}{3}<\log_{|G|}(q+1)\leq\frac{\log(q+1)}{\log(\frac{1}{2}q(q^2-1))}$, and for $q\to\infty$, this upper bound converges to $\frac{1}{3}$ strictly monotonously from above.

(ii) $G\cong\PSL_2(p)^2$ for some prime $p\geq 5$. In this case, $\Lambda(\Aut(G))=p(p+1)$, we have $\frac{1}{3}<\log_{|G|}(p(p+1))=\frac{\log(p(p+1))}{\log(\frac{1}{2}p(p^2-1))}$, and for $p\to\infty$, this upper bound converges to $\frac{1}{3}$ strictly monotonously from above.

(iii) $G\cong\PSL_2(p)^3$ for some prime $p\geq 5$. In this case, $\Lambda(\Aut(G))=\frac{1}{2}p(p^2-1)=|G|^{\frac{1}{3}}$.

\noindent (2) $\Lambda_{\aff}(\Aut(G))\leq|G|^{\frac{2}{3}}$, with the following exceptions: $G\cong\PSL_2(p)$ for some prime $p\geq 5$. In this case, $\Lambda_{\aff}(\Aut(G))=p(p+1)$, we have $\frac{2}{3}<\log_{|G|}(p(p+1))=\frac{\log(p(p+1))}{\log(\frac{1}{2}p(p^2-1))}$, and for $p\to\infty$, this upper bound converges to $\frac{2}{3}$ strictly monotonously from above.
\end{lemma}

The purpose of this section is to show how to deduce all the main results from Lemma \ref{mainLem}, so until the end of this section, the word \enquote{proof} means \enquote{proof conditional on Lemma \ref{mainLem}}. We first give the precise definition of the constants $E_1$ and $E_2$ from Theorem \ref{mainTheo1}:

\begin{notation}\label{eNot}
(1) We set $e_1:=\log_{60}(6)=0.437618\ldots$ and $E_1:=\frac{1}{e_1-1}=-1.778151\ldots$.

\noindent (2) We set $e_2:=\log_{60}(30)$ and $E_2:=\frac{1}{e_2-1}=-5.906890\ldots$.
\end{notation}

\begin{lemma}\label{auxiliaryLem1}
(1) For all finite nonabelian characteristically simple groups $G$, we have $\Lambda(\Aut(G))\leq|G|^{e_1}$, with equality if and only if $G\cong\PSL_2(5)\cong\Alt_5$.

\noindent (2) For every $\epsilon>0$, we have $\Lambda(\Aut(G))\leq|G|^{\frac{1}{3}+\epsilon}$ for almost all finite nonabelian characteristically simple groups $G$.

\noindent (3) For all finite nonabelian characteristically simple groups $G$, we have $\Lambda_{\aff}(\Aut(G))\leq|G|^{e_2}$, with equality if and only if $G\cong\PSL_2(5)\cong\Alt_5$.

\noindent (4) For every $\epsilon>0$, we have $\Lambda_{\aff}(\Aut(G))\leq|G|^{\frac{2}{3}+\epsilon}$ for almost all finite nonabelian characteristically simple groups $G$.
\end{lemma}

\begin{proof}
The statements in (2) and (4) follow immediately from Lemma \ref{mainLem}. For (1), note that by Lemma \ref{mainLem}(1), we have $\Lambda(\Aut(\PSL_2(5)))=6=|\PSL_2(5)|^{e_1}$, and using the strict monotonicity of the upper bounds in Lemma \ref{mainLem}(1), it is not difficult to see that this is the only case where equality holds. The proof of (2) is analogous.
\end{proof}

\begin{lemma}\label{auxiliaryLem2}
Let $H$ be a finite semisimple group. Then:

\noindent (1) $\Lambda(H)\leq|\Soc(H)|^{e_1}$.

\noindent (2) $\Lambda_{\aff}(H)\leq|\Soc(H)|^{e_2}$.
\end{lemma}

\begin{proof}
Let $S_1,\ldots,S_r$ be pairwise nonisomorphic nonabelian finite simple groups, $n_1,\ldots,n_r\in\mathbb{N}^+$ such that $\Soc(H)\cong S_1^{n_1}\times\cdots\times S_r^{n_r}$. Using the facts that $\Aut(H)$ embeds into $\Aut(\Soc(H))$, that $\Lambda(G)=\meo(\Aut(G))$ for all finite semisimple groups $G$ (Corollary \ref{regularCor}(1,i)) and that $\Lambda(R)=\meo(\Aut(R))=\Lambda(\Aut(R))$ for all finite centerless CR-groups $R$ (Corollary \ref{regularCor}(1,i) and (2,i)), we conclude that \[\Lambda(H)=\meo(\Aut(H))\leq\meo(\Aut(\Soc(H)))=\Lambda(\Soc(H))=\Lambda(S_1^{n_1}\times\cdots\times S_r^{n_r})\leq\]\[\leq\Lambda(S_1^{n_1})\cdots\Lambda(S_r^{n_r})=\Lambda(\Aut(S_1^{n_1}))\cdots\Lambda(\Aut(S_r^{n_r}))\leq|S_1|^{e_1n_1}\cdots|S_r|^{e_1n_r}=|\Soc(H)|^{e_1},\] where the last inequality follows from Lemma \ref{auxiliaryLem1}(1). This proves the inequality in (1). For (2), we use the fact that $H$ embeds into $\Aut(\Soc(H))$, that $\Lambda_{\aff}(G)=\meo(\Hol(G))$ for all finite semisimple groups $G$ (Corollary \ref{regularCor}(1,ii)) and that, by completeness of $\Aut(S_1^{n_1}\times\cdots\times S_r^{n_r})=\Aut(S_1^{n_1})\times\cdots\times\Aut(S_r^{n_r})$, the tuple $(\Aut(S_1^{n_1}),\ldots,\Aut(S_r^{n_r}))$ has the splitting property, to conclude, with one application of Lemma \ref{auxiliaryLem1}(3) at the end, that \[\Lambda_{\aff}(H)=\meo(\Hol(H))\leq\meo(\Hol(\Aut(\Soc(H))))=\Lambda_{\aff}(\Aut(\Soc(H)))=\]\[=\Lambda_{\aff}(\Aut(S_1^{n_1})\times\cdots\times\Aut(S_r^{n_r}))\leq\Lambda_{\aff}(\Aut(S_1^{n_1}))\cdots\Lambda_{\aff}(\Aut(S_r^{n_r}))\leq\]\[\leq|S_1|^{e_2n_1}\cdots|S_r|^{e_2n_r}=|\Soc(H)|^{e_2}.\]
\end{proof}

\begin{proof}[Proof of Theorem \ref{mainTheo1}]
For (1), using the assumption as well as Lemmata \ref{transferLem}(1) and \ref{auxiliaryLem2}(1), we conclude that $\rho\leq\lambda(G)\leq\lambda_{\aff}(\Rad(G))\cdot\lambda(G/\Rad(G))\leq1\cdot|G/\Rad(G)|^{e_1-1}$, and so $[G:\Rad(G)]\geq\rho^{\frac{1}{e_1-1}}$. The proof for (2) is analogous.
\end{proof}

\begin{proof}[Proof of Corollary \ref{mainCor}]
The statements about cycle lengths in $\Alt_5\cong\PSL_2(5)$ follow immediately from Lemma \ref{mainLem}. As for the two asserted implications:

\noindent For (1): By Theorem \ref{mainTheo1}(1) (and strict monotonicity of power functions), $\lambda(G)>\frac{1}{10}$ implies that $[G:\Rad(G)]<(\frac{1}{10})^{E_1}=60$, and thus that $[G:\Rad(G)]=1$.

\noindent For (2): This is similar to (1), but more involved. By Theorem \ref{mainTheo1}(2), $\lambda_{\aff}(G)>\frac{1}{4}$ implies that $[G:\Rad(G)]<(\frac{1}{4})^{E_2}=3600$. So if any nonsolvable finite group $G$ with $\lambda_{\aff}(G)>\frac{1}{4}$ existed, then $G/\Rad(G)$ would have socle a nonabelian finite simple group $S$ of order less than $3600$. By Lemma \ref{transferLem}(2), it would follow that $\lambda_{\aff}(S)>\frac{1}{4}$, so in order to get a contradiction, it suffices to check that $\lambda_{\aff}(S)\leq\frac{1}{4}$ for all nonabelian finite simple groups $S$ such that $|S|<3600$. By CFSG, there are precisely eight such $S$, namely $\PSL_2(q)$ for $q=5,7,9,8,11,13,17$ and $\Alt_7$. By Corollary \ref{regularCor}(1,ii), it is sufficient to compute $\frac{\meo(\Hol(S))}{|S|}$ for these eight $S$, which we did with the help of GAP \cite{GAP4}. For the $\PSL_2(q)$, the results are summarized in Table \ref{table1}, and we also got that $\lambda_{\aff}(\Alt_7)=\frac{1}{42}$:

\begin{table}[h]
\caption{$\lambda_{\aff}$-values of the nonabelian finite simple groups of order smaller than $3600$, excluding $\Alt_7$}\label{table1}
\begin{center}
\begin{tabular}{|c|c|c|c|c|c|c|c|c|}
\hline
$q$ & $5$ & $7$ & $9$ & $8$ & $11$ & $13$ & $17$ \\ \hline
$\lambda_{\aff}(\PSL_2(q))$ & $\frac{1}{4}$ & $\frac{1}{6}$ & $\frac{1}{9}$ & $\frac{1}{8}$ & $\frac{1}{10}$ & $\frac{1}{12}$ & $\frac{1}{16}$ \\ \hline
\end{tabular}
\end{center}
\end{table}
\end{proof}

For proving Theorem \ref{mainTheo2}, we introduce the following notation:

\begin{notation}\label{constantsNot}
(1) For $\kappa\in\left(0,\frac{2}{3}\right]$ and $\kappa_{\aff}\in\left(0,\frac{1}{3}\right]$, we denote by $\T^{(\kappa)}$ the set of finite nonabelian characteristically simple groups $T$ such that $\Lambda(\Aut(T))\geq|T|^{\frac{1}{3}+\kappa}$, and by $\T_{\aff}^{(\kappa_{\aff})}$ the set of finite nonabelian characteristically simple groups $T$ such that $\Lambda_{\aff}(\Aut(T))\geq|T|^{\frac{2}{3}+\kappa_{\aff}}$. Note that by Lemma \ref{mainLem}, $\T^{(\kappa)}$ and $\T_{\aff}^{(\kappa_{\aff})}$ are finite.

\noindent (2) For $\epsilon\in\left(0,\frac{2}{3}\right]$, $\epsilon_{\aff}\in\left(0,\frac{1}{3}\right]$ and $\rho\in\left(0,1\right)$, set

\[C^{(1)}(\epsilon,\rho):=\prod_{T\in\T^{(\rho\epsilon)}}{\frac{\Lambda(\Aut(T))}{|T|^{\frac{1}{3}+\rho\epsilon}}}\text{ and }C_{\aff}^{(1)}(\epsilon_{\aff},\rho):=\prod_{T\in\T_{\aff}^{(\rho\epsilon_{\aff})}}{\frac{\Lambda_{\aff}(\Aut(T))}{|T|^{\frac{2}{3}+\rho\epsilon_{\aff}}}},\]

\[C^{(2)}(\epsilon,\rho):=\prod_{T\in\T^{(\frac{1}{2}\rho\epsilon)}}{|T|}\text{ and }C_{\aff}^{(2)}(\epsilon_{\aff},\rho):=\prod_{T\in\T_{\aff}^{(\frac{1}{2}\rho\epsilon_{\aff})}}{|T|},\]

\[C(\epsilon,\rho):=C^{(1)}(\epsilon,\rho)^{\frac{1}{\rho/2\cdot\epsilon}}\cdot C^{(2)}(\epsilon,\rho)\text{ and }C_{\aff}(\epsilon_{\aff},\rho):=C_{\aff}^{(1)}(\epsilon_{\aff},\rho)^{\frac{1}{\rho/2\cdot\epsilon_{\aff}}}\cdot C_{\aff}^{(2)}(\epsilon_{\aff},\rho),\]

\[D(\epsilon,\rho):=\max\{|H|+1\mid H\text{ a finite semisimple group such that }|\Soc(H)|<C(\epsilon,\rho)\},\] and \[D_{\aff}(\epsilon_{\aff},\rho):=\max\{|H|+1\mid H\text{ a finite semisimple group such that }|\Soc(H)|<C_{\aff}(\epsilon_{\aff},\rho)\}.\]
\end{notation}

Theorem \ref{mainTheo2}(1) will follow rather easily from the following:

\begin{theorem}\label{auxiliaryTheo}
Let $\epsilon\in\left(0,\frac{2}{3}\right],\epsilon_{\aff}\in\left(0,\frac{1}{3}\right],\rho\in\left(0,1\right)$.

\noindent (1) Let $H$ be a finite semisimple group such that $|H|\geq D(\epsilon,\rho)$ holds. Then $\Lambda(H)\leq|\Soc(H)|^{\frac{1}{3}+\rho\epsilon}$.

\noindent (2) Let $H$ be a finite semisimple group such that $|H|\geq D_{\aff}(\epsilon_{\aff},\rho)$ holds. Then $\Lambda_{\aff}(H)\leq|\Soc(H)|^{\frac{2}{3}+\rho\epsilon_{\aff}}$.

\noindent (3) Let $G$ be a finite group such that $\Lambda(G)\geq|G|^{\frac{1}{3}+\epsilon}$. Then we have the following: $[G:\Rad(G)]\leq\max(D(\epsilon,\rho),|G|^{\frac{2/3-\epsilon}{2/3-\rho\epsilon}})$.

\noindent (4) Let $G$ be a finite group such that $\Lambda_{\aff}(G)\geq|G|^{\frac{2}{3}+\epsilon_{\aff}}$. Then we have the following: $[G:\Rad(G)]\leq\max(D_{\aff}(\epsilon_{\aff},\rho),|G|^{\frac{1/3-\epsilon_{\aff}}{1/3-\rho\epsilon_{\aff}}})$.
\end{theorem}

\begin{proof}
For (1): Let $S_1,\ldots,S_r$ be pairwise nonisomorphic nonabelian finite simple groups and $n_1,\ldots,n_r\in\mathbb{N}^+$ such that $\Soc(H)\cong S_1^{n_1}\times\cdots\times S_r^{n_r}$. For $i=1,\ldots,r$, set $T_i:=S_i^{n_i}$. Note that, as in the proof of Lemma \ref{auxiliaryLem2}(1), we have $\Lambda(H)\leq\Lambda(\Aut(T_1))\cdots\Lambda(\Aut(T_r))$.

The idea is the following: We will bound each $\Lambda(\Aut(T_i))$ from above by a power $|T_i|^{f_i}$, and we would be done if all $f_i$ were less than or equal to $\frac{1}{3}+\rho\epsilon$. In view of the exceptional cases in Lemma \ref{mainLem}, we cannot expect this to happen, but since by the same lemma, almost all finite nonabelian characteristically simple $T$ satisfy $\Lambda(\Aut(T))<|T|^{\frac{1}{3}+\frac{\rho\epsilon}{2}}$, and this upper bound has some capacity to \enquote{swallow} factors greater than $1$ and still remain smaller than $|T|^{\frac{1}{3}+\rho\epsilon}$, if the order of $|\Soc(H)|$ is large enough, the \enquote{swallowing capacity} of the factors $T_i\notin\T^{(\frac{\rho\epsilon}{2})}$ will be big enough to make up for the \enquote{spillover} of all \enquote{problematic} factors coming from the finite set $\T^{(\rho\epsilon)}$.

Formally, we proceed as follows. W.l.o.g., assume that there exist $k,l\in\mathbb{N}$ with $k+l\leq r$ such that $T_1,\ldots,T_k\in\T^{(\rho\epsilon)}$, $T_{k+1},\ldots,T_{k+l}\in\T^{(\frac{1}{2}\rho\epsilon)}\setminus\T^{(\rho\epsilon)}$ and $T_{k+l+1},\ldots,T_r\notin\T^{(\frac{1}{2}\rho\epsilon)}$. Note that by definition of $D(\epsilon,\rho)$ and the assumption, we have $|\Soc(H)|\geq C(\epsilon,\rho)$. By definition of $C_2(\epsilon,\rho)$, we have $|T_1|\cdots|T_{k+l}|\leq C_2(\epsilon,\rho)$, and so by definition of $C(\epsilon,\rho)$, we conclude that $|T_{k+l+1}|\cdots|T_r|=\frac{|\Soc(H)|}{|T_1|\cdots|T_{k+l}|}\geq\frac{C(\epsilon,\rho)}{C_2(\epsilon,\rho)}=C_1(\epsilon,\rho)^{\frac{1}{\rho/2\cdot\epsilon}}$. It follows that \[\Lambda(H)\leq\prod_{i=1}^k{\Lambda(\Aut(T_i))}\cdot\prod_{i=k+1}^{k+l}{\Lambda(\Aut(T_i))}\cdot\prod_{i=k+l+1}^r{\Lambda(\Aut(T_i))}\leq\]\[\leq\prod_{i=1}^k{|T_i|^{\frac{1}{3}+\rho\epsilon}}\cdot C_1(\epsilon,\rho)\cdot\prod_{i=k+1}^{k+l}{|T_i|^{\frac{1}{3}+\rho\epsilon}}\cdot\prod_{i=k+l+1}^r{|T_i|^{\frac{1}{3}+\rho\epsilon}}\cdot(\prod_{i=k+l+1}^r{|T_i|})^{-\frac{\rho\epsilon}{2}}\leq\]\[\leq|\Soc(H)|^{\frac{1}{3}+\rho\epsilon}\cdot C_1(\epsilon,\rho)\cdot(C_1(\epsilon,\rho)^{\frac{1}{\rho/2\cdot\epsilon}})^{-\frac{\rho\epsilon}{2}}=|\Soc(H)|^{\frac{1}{3}+\rho\epsilon}.\]

\noindent For (2): This is analogous to the proof of (1).

\noindent For (3): We show the contraposition: Assume that $G$ is a finite group such that $[G:\Rad(G)]>\max(D(\epsilon,\rho),|G|^{\frac{2/3-\epsilon}{2/3-\rho\epsilon}})$. We need to show that $\Lambda(G)<|G|^{\frac{1}{3}+\epsilon}$. Note that by (1), we have $\Lambda(G/\Rad(G))\leq|\Soc(G/\Rad(G))|^{\frac{1}{3}+\rho\epsilon}\leq|G/\Rad(G)|^{\frac{1}{3}+\rho\epsilon}$. It follows that \[\Lambda(G)\leq\Lambda_{\aff}(\Rad(G))\cdot\Lambda(G/\Rad(G))\leq|\Rad(G)|\cdot|G/\Rad(G)|^{\frac{1}{3}+\rho\epsilon}=\]\[=|\Rad(G)|^{\frac{2}{3}-\rho\epsilon}\cdot|G|^{\frac{1}{3}+\rho\epsilon}<(|G|^{1-\frac{2/3-\epsilon}{2/3-\rho\epsilon}})^{\frac{2}{3}-\rho\epsilon}\cdot|G|^{\frac{1}{3}+\rho\epsilon}=|G|^{\epsilon-\rho\epsilon}\cdot|G|^{\frac{1}{3}+\rho\epsilon}=|G|^{\frac{1}{3}+\epsilon}.\]

\noindent For (4): This is analogous to the proof of (3).
\end{proof}

\begin{proof}[Proof of Theorem \ref{mainTheo2}]
For (1): Set $K(\epsilon,\xi):=D(\epsilon,\frac{2/3\xi}{\epsilon(1-3/2\epsilon+\xi)})$. Then by setting $\rho:=\frac{2/3\xi}{\epsilon(1-3/2\epsilon+\xi)}$ in Theorem \ref{auxiliaryTheo}(2), we find that $\Lambda(G)\geq|G|^{\frac{1}{3}+\epsilon}$ implies \[[G:\Rad(G)]\leq\max(D(\epsilon,\rho),|G|^{\frac{2/3-\epsilon}{2/3-\rho\epsilon}})=\max(K(\epsilon,\xi),|G|^{1-3/2\epsilon+\xi}).\]

\noindent For (2): This is analogous to (1).

\noindent For (3): Denote by $p_n$ the $n$-th prime number (starting with $p_0=2$) and set $G_n:=\PGL_2(p_{n+2})=\Aut(\PSL_2(p_{n+2}))$. That this choice of $G_n$ does the job follows from Lemma \ref{mainLem}, since $\log_{p(p^2-1)}(p+1)>\frac{1}{3}$ and $\log_{p(p^2-1)}(p(p+1))>\frac{2}{3}$ for all primes $p\geq 5$.
\end{proof}

\section{Proof of the main lemma}\label{sec4}

We now tackle the final task of proving the main lemma, Lemma \ref{mainLem}. So let $G=S^n$, where $S$ is a nonabelian finite simple group and $n\in\mathbb{N}^+$. We make a case-by-case analysis using the CFSG. In most cases, Lemma \ref{boundLem}(2) will be sufficient to this end, but some cases require sharper upper bounds.

\subsection{Case: \texorpdfstring{$S$}{S} is sporadic}\label{subsec4P1}

Using Lemma \ref{boundLem}(2) and the information on $|S|,\Out(S)$ and $\meo(S)$ for sporadic $S$ from \cite{CCNPW85a}, this case only consists in some routine checks.

\subsection{Case: \texorpdfstring{$S=\mathcal{A}_m,m\geq7$}{S=Am, m>=7}}\label{subsec4P2}

Note that $\Alt_5\cong\PSL_2(5)$ and $\Alt_6\cong\PSL_2(9)$ will be treated in the next case. We also use Lemma \ref{boundLem}(2) here. That is, we want to show that $g(n)\cdot\meo(\Sym_m^n)<(\frac{1}{2}m!)^{n/3}$ for all $n\in\mathbb{N}^+$ and all $m\geq 7$.

For $n=1$, this is the inequality $g(m)<(\frac{1}{2}m!)^{1/3}$ for $m\geq 7$. But $g(m)<3^{m/3}$, and one easily verifies $3^{m/3}<(\frac{1}{2}m!)^{1/3}$ for all $m\geq 7$.

For $n=2$, the inequality turns into $2\cdot\meo(\Sym_m^2)<(\frac{1}{2}m!)^{2/3}$. Now $\meo(\Sym_m^2)<g(m)^2$, and it is easy to verify $2\cdot g(m)^2<(\frac{1}{2}m!)^{2/3}$ for $m\geq 7$.

For $n=3$, the inequality to show is $3\cdot\meo(\Sym_m^3)<\frac{1}{2}m!$, and it is easy to verify the stronger $3\cdot g(m)^3<\frac{1}{2}m!$.

Finally, for $n\geq 4$, we use the bound $g(n)\cdot\meo(\Sym_m^n)<3^{n/3}\cdot\Psi(m)<3^{n/3}\cdot\e^{1.03883\cdot m}$, which reduces the inequality to $\e^{1.03883\cdot m}<(\frac{1}{6}m!)^{n/3}$ for $n\geq 4$ and $m\geq 7$, and this is easy to verify.

\subsection{Case: \texorpdfstring{$S=\mathrm{PSL}_2(q),q\geq 5$}{S=PSL2(q), q>=5}}\label{subsec4P3}

This is the most complicated case, requiring to investigate the five subcases $n=1,2,3,4$ and $n\geq 5$. Recall that $\Aut(\PSL_2(q))=\PGL_2(q)\rtimes\Gal(\mathbb{F}_q/\mathbb{F}_p)$, and in particular, there is a natural embedding $\PSL_2(q)\hookrightarrow\PGL_2(q)$.

\subsubsection{Subcase: \texorpdfstring{$n=1$}{n=1}}\label{subsubsec4P3P1}

Our goal is to show the following:

\begin{theoremmm}\label{pslTheo}
Let $q\geq 5$ be primary. Then:

\noindent (1) $\Lambda(\Aut(\PSL_2(q)))=q+1$.

\noindent (2) $\Lambda_{\aff}(\Aut(\PSL_2(q)))=\begin{cases}q(q+1), & \text{if }q\text{ is prime}, \\ q^2-1, & \text{if }q\text{ is even}, \\ \frac{1}{2}(q^2-1), & \text{if }q\text{ is odd and not prime.}\end{cases}$.
\end{theoremmm}

By Corollary \ref{regularCor}(2,i), $\Lambda(\Aut(\PSL_2(q)))=\meo(\Aut(\PSL_2(q)))$, which was already determined by Guest, Morris, Praeger and Spiga in \cite{GMPS15a}, see Table 3 there. The following lemma is an extract from the proof of \cite[Theorem 2.16]{GMPS15a}:

\begin{lemmmma}\label{gmpsLem}
Let $q\geq 5$ be primary, with prime base $p$.

\noindent (1) Denote by $\pi:\Aut(\PSL_2(q))=\PGL_2(q)\rtimes\Gal(\mathbb{F}_q/\mathbb{F}_p)\rightarrow\Gal(\mathbb{F}_q/\mathbb{F}_p)$ the canonical projection. Let $\alpha\in\Aut(\PSL_2(q))$ such that $\ord(\pi(\alpha))=e$. Then $\ord(\alpha)\leq e\cdot (q^{1/e}+1)$.

\noindent (2) $\mao(\PSL_2(q))=q+1$.\qed
\end{lemmmma}

They proved point (2) using point (1) (whose proof used Lang-Steinberg maps). Since point (1) of Theorem \ref{pslTheo} is now clear, let us outline the strategy for proving point (2): By Theorem \ref{regularTheo}, we know that the largest cycle length of any periodic affine map $\A_{x,\alpha}$ of $\Aut(\PSL_2(q))$ coincides with its order, which is the product $\ord(\alpha)\cdot\ord(\sh_{\alpha}(x))$. By completeness of $\Aut(\PSL_2(q))$, we know that $\ord(\alpha)$ is an element order in $\Aut(\PSL_2(q))$, so the order of any periodic affine map of $\Aut(\PSL_2(q))$ is the product of two automorphism orders of $\PSL_2(q)$. If we know a list of the first few largest automorphism orders of $\PSL_2(q)$ which is long enough to ensure that for any periodic affine map whose order exceeds the asserted $\Lambda_{\aff}$-value, the two factor orders must be in the list, we can systematically go through the possible combinations, deriving a contradiction in each case using Lemmata \ref{divisorLem} and \ref{centralizerLem}. It will then remain to show that the asserted $\Lambda_{\aff}$-value is indeed the $\Lambda$-value of some periodic affine map of $\Aut(\PSL_2(q))$, which can be done by Lemma \ref{lcmLem}.

We can indeed extend the list of largest automorphism orders of $\PSL_2(q)$ to our needs in a way similar to how Guest, Morris, Praeger and Spiga derived point (2) of Lemma \ref{gmpsLem} from point (1):

\begin{lemmmma}\label{largestOrdLem}
(1) Let $q=2^f$ with $f\geq 3$. The two largest automorphism orders of $\PSL_2(q)$ are $q+1$ and $q-1$.

\noindent (2) Let $q=p^f\geq 5$ with $p$ an odd prime and $f\geq 1$.

(i) If $f=1$, then the five largest automorphism orders of $\PSL_2(q)$ are $q+1,q,q-1,\frac{q+1}{2},\frac{q-1}{2}$.

(ii) If $f\geq 2$ and $(p,f)\not=(3,2)$, then the four largest automorphism orders of $\PSL_2(q)$ are $q+1,q-1,\frac{q+1}{2},\frac{q-1}{2}$. Also, $\ord(\alpha)\leq\frac{q-1}{2}$ for any $\alpha\in\Aut(\PSL_2(q))\setminus\PGL_2(q)$, where the inequality is strict for $q\not=25$.

(iii) The four largest automorphism orders of $\PSL_2(9)\cong\Alt_6$ are $10,8,6,5$.
\end{lemmmma}

For those parts of the argument where we will use Lemma \ref{centralizerLem}, we will need some statements about centralizers in $\Aut(\PSL_2(q))$ for odd $q$:

\begin{lemmmma}\label{pslCentLem}
(1) Let $p\geq 5$ be prime, and let $\alpha\in\Aut(\PSL_2(p))=\PGL_2(p)$ be of order $p$. Then $\C_{\Aut(\PSL_2(p))}(\alpha)=\langle\alpha\rangle\subseteq\PSL_2(p)$.

\noindent (2) Let $q\geq 5$ be odd, primary, $q\notin\{9,25\}$, and let $\alpha\in\Aut(\PSL_2(q))$ be of order $\frac{q-1}{2}$. Then $\C_{\Aut(\PSL_2(q))}(\alpha)\subseteq\PGL_2(q)$.

\noindent (3) Let $q\geq 5$ be odd, primary, and let $\alpha\in\Aut(\PSL_2(q))$ be of order $q-1$. Then $\C_{\Aut(\PSL_2(q))}(\alpha)\subseteq\PGL_2(q)$.

\noindent (4) Let $q\geq 5$ be odd, primary, and let $\alpha\in\Aut(\PSL_2(q))$ be of order $\frac{q+1}{2}$. Then $\C_{\Aut(\PSL_2(q))}(\alpha)\subseteq\PGL_2(q)$.

\noindent (5) Let $q\geq 5$ be odd, primary, and let $\alpha\in\Aut(\PSL_2(q))$ be of order $q+1$. Then $\C_{\Aut(\PSL_2(q))}(\alpha)\subseteq\PGL_2(q)$.
\end{lemmmma}

Before proving Lemmata \ref{largestOrdLem} and \ref{pslCentLem}, for the readers' convenience, we quickly recall some basic facts on the element structure of $\PGL_2(q)$ for primary $q$ with prime base $p$. We denote by $\pi_0:\GL_2(q)\rightarrow\PGL_2(q)$ and $\pi_1:\GL_2(q^2)\rightarrow\PGL_2(q^2)$ the canonical projections.

\begin{enumerate}
\item Every element order in $\PGL_2(q)$ is a divisor of one of the following: $p,q+1,q-1$.

\item Every element in $\PGL_2(q)$ of order $p$ is conjugate in $\PGL_2(q)$ to an element of the form $\pi_0(\begin{pmatrix}1 & x \\ 0 & 1\end{pmatrix})$ with $x\in\mathbb{F}_q^{\ast}$. These elements are also in $\PSL_2(q)$.

\item Every element in $\PGL_2(q)$ of order a divisor of $q-1$ is conjugate in $\PGL_2(q)$ to an element of the form $\pi_0(\begin{pmatrix}1 & 0 \\ 0 & a\end{pmatrix})$ with $a\in\mathbb{F}_q^{\ast}$. Clearly, the order of such an element in $\PGL_2(q)$ equals the order of $a\in\mathbb{F}_q^{\ast}$, so all divisors of $q-1$ occur as element orders. Furthermore, such an element is in $\PSL_2(q)$ if and only if $a$ is a square in $\mathbb{F}_q$, whence for even $q$, all these elements are also in $\PSL_2(q)$, and for odd $q$, precisely those whose order is a divisor of $\frac{q+1}{2}$ are in $\PSL_2(q)$.

\item Every element in $\PGL_2(q)$ of order a divisor of $q+1$, but not of $q-1$, is conjugate in $\PGL_2(q^2)$ to an element of the form $\pi_1(\begin{pmatrix}1 & 0 \\ 0 & a\end{pmatrix})$ with $a\in\mathbb{F}_{q^2}\setminus\mathbb{F}_q$. As before, all such divisors of $q+1$ occur as element orders, and among such elements, precisely those where $a$ is a square in $\mathbb{F}_{q^2}$ are in $\PSL_2(q)$, so again, for even $q$, all such elements are also in $\PSL_2(q)$, and for odd $q$, precisely those whose order is a divisor of $\frac{q+1}{2}$ are also in $\PSL_2(q)$.
\end{enumerate}

\begin{proof}[Proof of Lemma \ref{largestOrdLem}]
Denote by $\pi:\Aut(\PSL_2(q))\rightarrow\Gal(\mathbb{F}_q/\mathbb{F}_p)$ the canonical projection. 

\noindent For (1): That $q+1$ is the largest automorphism order is just a special case of Lemma \ref{gmpsLem}(2), and $q-1$ is an automorphism order by the above facts on the element structure of $\PGL_2(q)$. It remains to show that $q=2^f$ is not an automorphism order, which goes as follows: If $\alpha\in\Aut(\PSL_2(q))$ had order $2^f$, then $2^f=\ord(\alpha)=\ord(\pi(\alpha))\cdot\ord(\alpha^{\ord(\pi(\alpha))})$. Now by the element structure, the only element orders in $\PGL_2(q)$ which are powers of $2$ are $1$ and $2$, and so $\ord(\alpha^{\ord(\pi(\alpha))})\leq 2$, and thus $\ord(\pi(\alpha))\geq 2^{f-1}$. But $\ord(\pi(\alpha))\mid|\Gal(\mathbb{F}_q/\mathbb{F}_2)|=f$, a contradiction.

\noindent For (2,i): Since $\Aut(\PSL_2(q))=\PGL_2(q)$ if $q$ is prime, the statement follows from the element structure of $\PGL_2(q)$.

\noindent For (2,ii): Again, by the element structure of $\PGL_2(q)$, the four listed numbers are certainly the four largest element orders in $\PGL_2(q)$, so it suffices to prove the second part of the claim. Let $\alpha\in\Aut(\PSL_2(q))\setminus\PGL_2(q)$, so that $e:=\ord(\pi(\alpha))>1$. We need to show that $\ord(\alpha)\leq\frac{q-1}{2}$, and actually $\ord(\alpha)<\frac{q-1}{2}$ unless $q=25$. By Lemma \ref{gmpsLem}(1), it is sufficient to show that $e(q^{1/e}+1)<\frac{q-1}{2}$ for $q\not=25$ (and to check that for $q=25$, where $e=2$, the left-hand side is equal to the right-hand side). For $q\not=25$ (i.e., $q\geq 27$), note that it suffices to show \begin{equation}\label{Eq1}\frac{4}{3}eq^{1/e}\leq\frac{13}{27}q,\end{equation} since \[e(q^{1/e}+1)=eq^{1/e}(1+\frac{1}{q^{1/e}})\leq\frac{4}{3}eq^{1/e},\] and \[\frac{q-1}{2}=q(\frac{1}{2}-\frac{1}{2q})\geq q(\frac{1}{2}-\frac{1}{54})=\frac{13}{27}q.\] (\ref{Eq1}) is equivalent to \[q\geq(\frac{36}{13}e)^{1+\frac{1}{e-1}},\] which is easy to verify in the case distinction $e=2$ (where $q\geq 49$) versus $e\geq 3$ (using that $q\geq 3^e$).

\noindent For (2,iii): This is readily checked with GAP \cite{GAP4}.
\end{proof}

We remark that, as is easy to check with GAP \cite{GAP4}, $\PSL_2(25)$ actually has automorphisms of order $12=\frac{25-1}{2}$ that are not in $\PGL_2(25)$.

\begin{proof}[Proof of Lemma \ref{pslCentLem}]
For (1): By the element structure of $\PGL_2(p)$, we have $\alpha\in\PSL_2(p)$, and $\alpha$ is conjugate in $\PGL_2(p)$ to an element of the form $\pi_0(\begin{pmatrix}1 & x \\ 0 & 1\end{pmatrix})$ for some $x\in\mathbb{F}_p^{\ast}$, so it suffices to prove the assertion for all such elements. However, since they are powers of one another, it actually suffices to show the assertion for $\alpha=\pi_0(\begin{pmatrix}1 & 1 \\ 0 & 1\end{pmatrix})$. So let $\begin{pmatrix}a & b \\ c & d\end{pmatrix}\in\GL_2(p)$ such that \begin{equation}\label{Eq2}\pi_0(\begin{pmatrix}1 & 1 \\ 0 & 1\end{pmatrix}\cdot\begin{pmatrix}a & b \\ c & d\end{pmatrix}\cdot\begin{pmatrix}1 & -1 \\ 0 & 1\end{pmatrix})=\pi_0(\begin{pmatrix}a & b \\ c & d\end{pmatrix}).\end{equation} (\ref{Eq2}) is equivalent to the existence of some $\lambda\in\mathbb{F}_p^{\ast}$ such that \begin{equation}\label{Eq3}\begin{pmatrix}a+c & b+d-a-c \\ c & d-c\end{pmatrix}=\lambda\cdot\begin{pmatrix}a & b \\ c & d\end{pmatrix}.\end{equation} If $\lambda\not=1$, then a comparison of the bottom left entries in (\ref{Eq3}) implies $c=0$ and thus also $a=0$, a contradiction. So $\lambda=1$, turning (\ref{Eq3}) into a system of linear equations over $\mathbb{F}_p$ which one checks to be equivalent to $c=0,a=d$. It follows that \[\pi_0(\begin{pmatrix}a & b \\ c & d\end{pmatrix})=\pi_0(\begin{pmatrix}a & b \\ 0 & a\end{pmatrix})=\pi_0(\begin{pmatrix}1 & b/a \\ 0 & 1\end{pmatrix})\in\langle\alpha\rangle.\]

\noindent For (2): Note that by Lemma \ref{largestOrdLem}(2,ii), $\alpha$ is an element of $\PGL_2(q)$, and so by the element structure of $\PGL_2(q)$, $\alpha$ is conjugate in $\PGL_2(q)$ to an element of the form $\pi_0(\begin{pmatrix}1 & 0 \\ 0 & x\end{pmatrix})$ with $x\in\mathbb{F}_q^{\ast}$ of order $\frac{q-1}{2}$ (i.e., $x$ generates the subgroup of squares in $\mathbb{F}_q^{\ast}$); it suffices to show that the centralizers in $\Aut(\PSL_2(q))$ of such elements are contained in $\PGL_2(q)$. We do so by contradiction: Assume that for some nontrivial field automorphism $\sigma=\Frob^e$ of $\mathbb{F}_q$, where $\Frob$ denotes the Frobenius automorphism of $\mathbb{F}_q$ and $1\leq e<f$, and for some $A=\begin{pmatrix}a & b \\ c & d\end{pmatrix}\in\GL_2(q)$, we have \begin{equation}\label{Eq4}\pi_0(A\sigma\cdot\begin{pmatrix}1 & 0 \\ 0 & x\end{pmatrix}\cdot\sigma^{-1}A^{-1})=\pi_0(\begin{pmatrix}1 & 0 \\ 0 & x\end{pmatrix}).\end{equation} Easy computations reveal that (\ref{Eq4}) is equivalent to the existence of some $\lambda\in\mathbb{F}_q^{\ast}$ such that \begin{equation}\label{Eq5}\frac{1}{ad-bc}\cdot\begin{pmatrix}ad-\sigma(x)bc & ab(\sigma(x)-1) \\ cd(1-\sigma(x)) & \sigma(x)ad-bc\end{pmatrix}=\lambda\cdot\begin{pmatrix}1 & 0 \\ 0 & x\end{pmatrix}.\end{equation} Comparing the coefficients in the bottom left and top right corners in (\ref{Eq5}), we find that $ab=0$ and $cd=0$, so either $a=d=0$ or $b=c=0$. In the latter case, comparing the coefficients in the top left corners of (\ref{Eq5}) yields $\lambda=1$, and thus comparing the bottom right coefficients in (\ref{Eq5}), we get that $\sigma(x)=x$, which implies $\frac{p^f-1}{2}\mid p^e-1$, or $p^f-1\mid 2(p^e-1)$, although $p^f-1>p^f-p=p\cdot(p^{f-1}-1)>2\cdot(p^e-1)$, a contradiction. In the first case, comparing the coefficients in the top left corners of (\ref{Eq5}) gives $\lambda=\sigma(x)$, and thus by comparing the coefficients in the bottom right corners of (\ref{Eq5}), $\sigma(x)=x^{-1}$, which implies $\frac{p^f-1}{2}\mid p^e+1$, or $p^f-1\mid 2(p^e+1)$, although it is easy to check that $2(p^e+1)\leq 2(p^{f-1}+1)<p^f-1$, a contradiction.

\noindent For (3): This can be treated with an argument analogous to the one for (2) (of course, except for the cases $q=9,25$, the statement immediately follows from (2)).

\noindent For (4): Consider the natural embedding \[\Aut(\PSL_2(q))=\PGL_2(q)\rtimes\Gal(\mathbb{F}_q/\mathbb{F}_p)\hookrightarrow\PGL_2(q^2)\rtimes\Gal(\mathbb{F}_{q^2}/\mathbb{F}_p)=\Aut(\PSL_2(q^2))\] extending the natural embedding $\PGL_2(q)\hookrightarrow\PGL_2(q^2)$, by means of which we view $\Aut(\PSL_2(q))$ as a subgroup of $\Aut(\PSL_2(q^2))$. By Lemma \ref{largestOrdLem}(2,ii), $\alpha\in\PGL_2(q)$, and by the element structure of $\PGL_2(q)$, $\alpha$ is conjugate in $\PGL_2(q^2)$ to an element of the form $\pi_1(\begin{pmatrix}1 & 0 \\ 0 & x\end{pmatrix})$, where the order of $x\in\mathbb{F}_{q^2}^{\ast}$ is $\frac{q+1}{2}$. Denote by $\Frob$ the Frobenius automorphism of $\mathbb{F}_{q^2}$. It is sufficient to show that $\C_{\Aut(\PSL_2(q^2))}(\pi_1(\begin{pmatrix}1 & 0 \\ 0 & x\end{pmatrix}))\subseteq\PGL_2(q^2)\rtimes\langle\Frob^f\rangle$, since this implies $\C_{\Aut(\PSL_2(q^2))}(\alpha)\subseteq\PGL_2(q^2)\rtimes\langle\Frob^f\rangle$, and so \[\C_{\Aut(\PSL_2(q))}(\alpha)=\C_{\Aut(\PSL_2(q^2))}(\alpha)\cap\Aut(\PSL_2(q))\subseteq\]\[\subseteq(\PGL_2(q^2)\rtimes\langle\Frob^f\rangle)\cap\Aut(\PSL_2(q))=\PGL_2(q).\] To see that among the elements of $\Aut(\PSL_2(q^2))$, $\pi_1(\begin{pmatrix}1 & 0 \\ 0 & x\end{pmatrix})$ only commutes with elements from $\PGL_2(q^2)\rtimes\langle\Frob^f\rangle$, we proceed by contradiction, with the same ansatz as in point (2). This time, the divisibility relations at which one arrives in the two cases are $p^f+1\mid 2(p^e-1)$ and $p^f+1\mid 2(p^e+1)$ respectively. Note that now, $1\leq e<2f$, so we cannot argue as in point (2) that the supposed multiple is always smaller than the supposed divisor. However, this idea at least excludes the case $e<f$, so we may write $e=f+k$ with $0\leq k<f$. Then it is easy to check that $2p^k-1<\frac{2(p^e-1)}{p^f+1}<2p^k$, making the first inequality contradictory. Similarly, one can exclude the case $k>0$ for the second inequality, leaving only the case $k=0$, i.e., $e=f$.

\noindent For (5): This follows immediately from (4).
\end{proof}

\begin{proof}[Proof of Theorem \ref{pslTheo}]
As pointed out before, point (1) of the theorem follows from Lemma \ref{gmpsLem}(2), so we focus on the proof of point (2). Let $A=\A_{x,\alpha}\in\Aff(\Aut(\PSL_2(q)))$ be such that $\Lambda(A)=\Lambda_{\aff}(\Aut(\PSL_2(q)))$. Set $o_1:=\ord(\alpha)$ and $o_2:=\ord(\sh_{\alpha}(x))$, so that $\Lambda(A)=\ord(A)=o_1\cdot o_2$, and note that $o_1,o_2\leq q+1$.

If $q$ is prime, then on the one hand, we cannot have $o_1=o_2=q+1$, since that would imply by Lemma \ref{divisorLem} that $(q+1)^2\mid|\Aut(\PSL_2(q))|=|\PGL_2(q)|=q(q^2-1)$, a contradiction. The next smaller potential order of $A$ is $q(q+1)$, which is indeed attained by Lemma \ref{lcmLem} and the fact that $\Aut(\PSL_2(q))=\PGL_2(q)$ contains both an element of order $q$ and of order $q+1$.

If $q=2^f$ with $f\geq 3$, then Lemma \ref{divisorLem} again excludes the case $o_1=o_2=q+1=2^f+1$. By Lemma \ref{largestOrdLem}(1), the next smaller potential order of $A$ is $(q+1)\cdot(q-1)=q^2-1$, which can be attained in view of Lemma \ref{lcmLem}.

Finally, consider the case $q=p^f$ with $p$ an odd prime and $f\geq 2$. First, one verifies with GAP \cite{GAP4} that $\Lambda_{\aff}(\Aut(\PSL_2(9)))=40=\frac{1}{2}(9^2-1)$ and $\Lambda_{\aff}(\Aut(\PSL_2(25)))=312=\frac{1}{2}(25^2-1)$, so we may assume $(p,f)\notin\{(3,2),(5,2)\}$ from now on. By the element structure of $\PGL_2(q)$ and Lemma \ref{lcmLem}, it is clear that $\frac{1}{2}(q^2-1)$ can be attained as the order of some periodic affine map of $\Aut(\PSL_2(q))$, so it remains to show that $o_1\cdot o_2\leq\frac{1}{2}(q^2-1)$. We do this in a case distinction.

First assume that $o_1=q+1$, so that by Lemma \ref{largestOrdLem}(2,ii), $\alpha\in\PGL_2(q)$. Then the inequality is equivalent to $o_2\leq\frac{q-1}{2}$. If $o_2>\frac{q-1}{2}$, by Lemma \ref{largestOrdLem}(2,ii) again, it follows that $o_2\in\{q+1,q-1,\frac{q+1}{2}\}$. In each of these three cases, using Lemma \ref{centralizerLem} and Lemma \ref{pslCentLem}(5,3,4) respectively, we conclude that $x\in\PGL_2(q)$. This gives a contradiction when $o_2=q+1$ or $o_2=q-1$, since by the fact that $[\PGL_2(q):\PSL_(q)]=2$ and $o_1$ is even, we get that $\sh_{\alpha}(x)\in\PSL_2(q)$, but $\PSL_2(q)$ does not have any elements of order $q+1$ or $q-1$. The case $o_2=\frac{q+1}{2}$ can be refuted by Lemma \ref{divisorLem} (applied to $G:=\PGL_2(q)$) again.

Next assume that $o_1=q-1$, in which case $\alpha\in\PGL_2(q)$ as well. The inequality is equivalent to $o_2\leq\frac{q+1}{2}$, so it remains to exclude the two cases $o_2=q+1$ and $o_2=q-1$, which can be done as in the previous case, deriving the contradictory $\sh_{\alpha}(x)\in\PSL_2(q)$.

If $o_1=\frac{q+1}{2}$, we only need to exclude the case $o_2=q+1$, which can be done as in the case $o_1=q+1$ using Lemma \ref{divisorLem}. Finally, if $o_1\leq\frac{q-1}{2}$, then the inequality holds for sure.
\end{proof}

Theorem \ref{pslTheo} implies by some easy computations that for primary $q\geq 5$, $\Lambda_{\aff}(\PSL_2(q))>|\PSL_2(q)|^{\frac{2}{3}}$ if and only if $q$ is a prime, in which case $\Lambda_{\aff}(\PSL_2(q))=q(q+1)$, and verification of the statement about monotonous convergence of the upper bound is also easy. This settles our discussion of the subcase $n=1$.

\subsubsection{Useful observations for the other subcases}\label{subsubsec4P3P2}

The following lemma is immediate from the element structure of $\PGL_2(p)$:

\begin{lemmmma}\label{affineDivLem}
Let $p\geq 5$ be a prime, and let $A\in\Aff(\Aut(\PSL_2(p)))=\Aff(\PGL_2(p))$. Then $\ord(A)$ is a divisor of one of the following: $p(p+1),p(p-1),p^2-1$.\qed
\end{lemmmma}

Another useful observation (similar in spirit to Lemma \ref{boundLem}(2)) is the following: Since we have $\Lambda(\Aut(\PSL_2(q)^n))=\meo(\Aut(\PSL_2(q)^n))$, and $\Lambda_{\aff}(\Aut(\PSL_2(q)^n))\leq\meo(\Aut(\PSL_2(q)^n))^2$, whenever $\Lambda(\Aut(\PSL_2(q)^n))\leq|\PSL_2(q)|^{\frac{n}{3}}$, we also have $\Lambda_{\aff}(\Aut(\PSL_2(q)^n))\leq|\PSL_2(q)|^{\frac{2n}{3}}$.

\subsubsection{Subcase: \texorpdfstring{$n=2$}{n=2}}\label{subsubsec4P3P3}

Clearly, for primes $p\geq 5$, $\Lambda(\Aut(\PSL_2(p)^2))=\Lambda(\Aut(\PSL_2(p))\wr\Sym_2)$ is bounded from below by $p(p+1)=\meo(\Aut(\PSL_2(p))^2)$, and by Lemma \ref{wreathLem}, elements from $\Aut(\PSL_2(p)^2)\setminus\Aut(\PSL_2(p))^2$ have order bounded from above by $2\cdot(p+1)<p(p+1)$, so indeed, we have $\Lambda(\Aut(\PSL_2(p)^2))=\meo(\Aut(\PSL_2(p)^2))=p(p+1)$. As for $q\geq 5$ that are not prime, we first verify directly with GAP \cite{GAP4} that $\meo(\Aut(\PSL_2(9)^2))=40<360^{2/3}$. For all other odd $q$, we can use Lemma \ref{largestOrdLem}(2,ii) to see that $\meo(\Aut(\PSL_2(q)^2))=\frac{1}{2}(q^2-1)<(\frac{1}{2}q(q^2-1))^{\frac{2}{3}}$, and Lemma \ref{wreathLem} to treat automorphisms outside $\Aut(\PSL_2(q)^2)$ as before. Finally, for $q=2^f$ with $f\geq 3$, by Lemma \ref{largestOrdLem}(1), we have $\meo(\Aut(\PSL_2(q)^2))=q^2-1<(q(q^2-1))^{\frac{2}{3}}$, and we can treat all other automorphisms by Lemma \ref{wreathLem} again.

As for $\Lambda_{\aff}$-values in the subcase $n=2$, by the \enquote{useful observation} after Lemma \ref{affineDivLem}, it remains to show that $\Lambda_{\aff}(\Aut(\PSL_2(p)^2))\leq|\PSL_2(p)|^{\frac{4}{3}}$ for primes $p\geq 5$. It is easily checked with GAP \cite{GAP4} that $\Lambda_{\aff}(\Aut(\PSL_2(5)^2))=120<60^{\frac{4}{3}}$, so we may assume $p\geq 7$ from now on. Let $A=\A_{x,\alpha}\in\Aff(\Aut(\PSL_2(p)^2))$. We know that we can identify $\alpha$ with an element in $\Aut(\PSL_2(p)^2)$, that $\meo(\Aut(\PSL_2(p))^2)=p(p+1)$ and that elements from $\Aut(\PSL_2(p)^2)\setminus\Aut(\PSL_2(p))^2$ have at most the order $2\cdot(p+1)$. Therefore, if not both $\alpha,\sh_{\alpha}(x)\in\Aut(\PSL_2(p))^2$, then the order of $A$ is at most $2(p+1)\cdot p(p+1)<(\frac{1}{2}p(p^2-1))^{\frac{4}{3}}$. So we may assume $\alpha,\sh_{\alpha}(x)\in\Aut(\PSL_2(p))^2$ from now on, and also $\ord(A)>2(p+1)\cdot p(p+1)$. The latter implies that the two components of $\sh_{\alpha}(x)$ must be of different order. But conjugation of $\sh_{\alpha}(x)$ by any element from $\Aut(\PSL_2(p)^2)\setminus\Aut(\PSL_2(p)^2)$ swaps the orders of the components, and so $\sh_{\alpha}(x)$ cannot commute with any such element. In other words, $\C_{\Aut(\PSL_2(p)^2)}(\sh_{\alpha}(x))\subseteq\Aut(\PSL_2(p))^2$, and so, by an application of Lemma \ref{centralizerLem}, we conclude that $x\in\Aut(\PSL_2(p))^2$. Together with $\alpha\in\Aut(\PSL_2(p))^2$, this implies that $A$ decomposes as a product $A_1\times A_2$, with $A_1,A_2\in\Aff(\Aut(\PSL_2(p)))$. Therefore, by Lemma \ref{affineDivLem}, $\ord(A)=\lcm(\ord(A_1),\ord(A_2))\leq p(p^2-1)<(\frac{1}{2}p(p^2-1))^{\frac{4}{3}}$.

\subsubsection{Subcase: \texorpdfstring{$n=3$}{n=3}}\label{subsubsec4P3P4}

Denote by $\pi_3:\Aut(\PSL_2(q)^3)=\Aut(\PSL_2(q))\wr\Sym_3\rightarrow\Sym_3$ the canonical projection. By a simple case distinction according to the cycle type of $\pi_3(\alpha)$, Lemma \ref{wreathLem} can be used to show that automorphisms $\alpha$ outside $\Aut(\PSL_2(q))^3$ have order bounded from above by $2q(q+1)<|\PSL_2(q)|$ in all cases. If $q$ is a prime, then since the element orders in $\Aut(\PSL_2(q))=\PGL_2(q)$ are just the divisors of $q+1,q$ and $q-1$, we have $\meo(\Aut(\PSL_2(q))^3)=\lcm(q+1,q,q-1)=\frac{1}{2}q(q^2-1)=|\PSL_2(q)|^{\frac{3}{3}}$. If $q=2^f$ with $f\geq 3$, by Lemma \ref{largestOrdLem}(1), we have $\meo(\Aut(\PSL_2(q)^3))<(q+1)(q-1)^2<|\PSL_2(q)|$. For $q=9$, one checks with GAP \cite{GAP4} that $\meo(\Aut(\PSL_2(9)^3))=120<360$, and for odd $q\geq 25$, using Lemma \ref{largestOrdLem}(2,ii), we conclude that $\meo(\Aut(\PSL_2(q))^3)<\frac{1}{2}(q+1)(q-1)^2<|\PSL_2(q)|$.

\subsubsection{Subcase: \texorpdfstring{$n=4$}{n=4}}\label{subsubsec4P3P5}

We will show $\Lambda(\Aut(\PSL_2(q)^4))<|\PSL_2(q)|^{\frac{4}{3}}$ for all primary $q\geq 5$. For $q=5$, one can check directly that $\meo(\Aut(\PSL_2(5))^4)=60<60^{\frac{4}{3}}$, and automorphisms $\alpha$ from outside $\Aut(\PSL_2(5))^4$ are treated with Lemma \ref{wreathLem} like before. Assuming $q\geq 7$, we have $\meo(\Aut(\PSL_2(q)^4))\leq g(4)\cdot\exp(\Aut(\PSL_2(q)))\leq 4\cdot \log_p(q)\cdot p\cdot\frac{q^2-1}{\gcd(2,q-1)}\leq 4\cdot |\PSL_2(q)|<|\PSL_2(q)|^{\frac{4}{3}}$.

\subsubsection{Subcase: \texorpdfstring{$n\geq 5$}{n>=5}}\label{subsubsec4P3P6}

Here we can use crude upper bounds and \enquote{get away with it}; it is sufficient and easy to verify that \[\Lambda(\Aut(\PSL_2(q)^n))\leq g(n)\cdot \exp(\Aut(\PSL_2(q)))<3^{\frac{n}{3}}\cdot|\PSL_2(q)|\leq|\PSL_2(q)|^{n/3}.\]

\subsection{Case: \texorpdfstring{$S=\PSL_d(q),d\geq 3,q\geq 2$}{S=PSLd(q), d>=3, q>=2}}\label{subsec4P4}

From now on, we will always work with Lemma \ref{boundLem}(2). Furthermore, we will use the information on maximum automorphism orders of finite simple groups from \cite[Table 3]{GMPS15a}. Note that since $\PSL_3(2)\cong\PSL_2(7)$, we may assume that $(d,q)\not=(3,2)$, and so $\mao(\PSL_d(q))=\frac{q^d-1}{q-1}$. In view of $\meo(\Aut(\PSL_d(q))^n)\leq\meo(\Aut(\PSL_d(q)))^n$, our goal is to show that \begin{equation}\label{eq2}g(n)\cdot\meo(\Aut(\PSL_d(q)))^n<|\PSL_d(q)|^{\frac{n}{3}}=(\frac{q^{d(d-1)/2}}{\gcd(d,q-1)}\cdot\prod_{i=2}^d{(q^i-1)})^{\frac{n}{3}}.\end{equation}

\subsubsection{Subcase: \texorpdfstring{$d=3$}{d=3}}\label{subsubsec4P4P1}

We need to treat the subsubcases $q=3$ and $q=4$ separately. Using GAP \cite{GAP4}, one finds that the list of element orders in $\Aut(\PSL_3(3))$ is $1,2,3,4,6,8,12,13$. This implies that \[g(1)\cdot\meo(\Aut(\PSL_3(3))^1)=1\cdot 13<5616^{\frac{1}{3}},\] that \[g(2)\cdot\meo(\Aut(\PSL_3(3)^2)=2\cdot 156<5616^{\frac{2}{3}},\] and that $g(n)\cdot\meo(\Aut(\PSL_3(3)^n))=g(n)\cdot 312<5616^{n/3}$ for $n\geq 3$. The subsubcase $q=4$ is treated analogously.

For $q\geq 5$, using Proposition \ref{landauProp}, we see that for proving (\ref{eq2}), it is sufficient to show \[(q^2+q+1)^2<\frac{q^3}{3\gcd(3,q-1)}\cdot(q-1)(q^2-1),\] which is easy to verify.

\subsubsection{Subcase: \texorpdfstring{$d=4$}{d=4} or \texorpdfstring{$d=5$}{d=5}}\label{subsubsec4P4P2}

For $d=4$, splitting the factor $(q^2)^{\frac{n}{3}}$ from the beginning of the right-hand side of (\ref{eq2}), we can \enquote{swallow} the factor $g(n)$ on the left-hand side by Proposition \ref{landauProp}, and see that it is sufficient to show that \begin{equation}\label{eq6}(q^3+q^2+q+1)^2<\frac{q^4}{\gcd(4,q-1)}(q-1)(q^3-1)(q^2-1).\end{equation} Replacing the left-hand side of (\ref{eq6}) by the larger $q^8$, dividing both sides by $q^8$ and performing appropriate cancelations and distributions of factors $q$ among the factors on the right-hand side, we get the stronger inequality \begin{equation}\label{eq8}1<\frac{1}{\gcd(4,q-1)}(q-1)\cdot(\sqrt{q}-\frac{1}{q^{5/2}})\cdot(\sqrt{q}-\frac{1}{q^{3/2}}),\end{equation} which is obviously true. The subcase $d=5$ can be treated in a similar way.

\subsubsection{Subcase: \texorpdfstring{$d\geq 6$}{d>=6}}\label{subsubsec4P4P3}

One can check that $2d\leq\frac{d(d-1)}{2}-2$ for $d\geq 6$. The left-hand side of (\ref{eq2}) is therefore bounded from above by \[(q^2)^{\frac{n}{3}}\cdot(q^d-1)^n<(q^2)^{\frac{n}{3}}\cdot(q^d-1)^{\frac{n}{3}}\cdot (q^{2d})^{\frac{n}{3}}\leq(q^2)^{\frac{n}{3}}\cdot(q^d-1)^{\frac{n}{3}}\cdot (q^{\frac{d(d-1)}{2}-2})^{\frac{n}{3}}=(q^{d(d-1)/2}\cdot(q^d-1))^{\frac{n}{3}},\] which is obviously smaller than the right-hand side of (\ref{eq2}).

\subsection{Case: \texorpdfstring{$S=\PSU_d(q),d\geq3,(d,q)\not=(3,2)$}{S=PSUd(q),d>=3,(d,q)!=(3,2)}}\label{subsec4P5}

Note that $|\PSU_d(q)|=\frac{1}{\gcd(d,q+1)}q^{d(d-1)/2}\prod_{i=2}^d{(q^i-(-1)^i)}$.

\subsubsection{Subcase: \texorpdfstring{$d=3$}{d=3}}\label{subsubsec4P5P1}

It follows from \cite[Table 3]{GMPS15a} that $\mao(\PSU_3(q))\leq q^2+q$, and so it is sufficient to show that \begin{equation}\label{eq9}g(n)\cdot(q^2+q)^n<(\frac{1}{\gcd(3,q+1)}q^3(q^3+1)(q^2-1))^{\frac{n}{3}}.\end{equation} Splitting $q^{\frac{n}{3}}$ from the right-hand side of (\ref{eq9}) to \enquote{swallow} $g(n)$, we see that (\ref{eq9}) is implied by \begin{equation}\label{eq10}\gcd(3,q+1)<\frac{q^2-q+1}{q+1}\cdot(1-\frac{1}{q}).\end{equation} For $q\geq 7$, the first factor on the right-hand side of (\ref{eq10}) is bounded from below by $4$, and so the entire right-hand side is bounded from below by $4\cdot\frac{6}{7}>3\geq\gcd(3,q+1)$. For $q=3$ and $q=4$, one verifies the validity of (\ref{eq10}) directly. Finally, for $q=5$, one can check that \[g(n)\cdot\meo(\Aut(\PSU_3(5))^n)<126000^{\frac{n}{3}},\] like we did for $q=3$ in the subcase $d=3$ of the previous case (Subsubsection \ref{subsubsec4P4P1}).

\subsubsection{Subcase: \texorpdfstring{$d\geq4$}{d>=4}}\label{subsubsec4P5P2}

We read off from \cite[Table 3]{GMPS15a} that $\mao(\PSU_4(q))\leq q^3+4$, so for the subsubcase $d=4$, we want to show that \begin{equation}\label{eq11}g(n)\cdot(q^3+4)^n<(\frac{q^6}{\gcd(4,q+1)}(q^4-1)(q^3+1)(q^2-1))^{\frac{n}{3}}.\end{equation} Splitting $(q+1)^{\frac{n}{3}}$ from the right-hand side of (\ref{eq11}) to \enquote{swallow} $g(n)$, we see that (\ref{eq11}) is weaker than \[(q^3+4)^3<\frac{q^6}{\gcd(4,q+1)}(q^4-1)(q^3+1)(q-1),\] which is easy to prove for all $q\geq 2$. The subsubcase $d=5$ is similar to $d=4$, using that $\mao(\PSU_5(q))<q^5$. Finally, using $\mao(\PSU_d(q))<q^d$, we can treat the subsubcase $d\geq 6$ similarly to the subcase $d\geq 6$ of the previous case (Subsubsection \ref{subsubsec4P4P3}).

\subsection{Case: \texorpdfstring{$S=\PSp_{2m}(q),m\geq 2,(m,q)\not=(2,2)$}{S=PSp2m(q),m>=2,(d,q)!=(2,2)} or \texorpdfstring{$S=\PO_{2m+1}(q),m\geq 3$}{S=POmega2m+1(q),m>=3}}\label{subsec4P6}

By \cite[Table 3]{GMPS15a}, in both cases, $\mao(S)\leq\frac{q^{m+1}}{q-1}$. Also, $|S|=\frac{q^{m^2}}{\gcd(2,q-1)}\prod_{i=1}^m{(q^{2i}-1)}$ in both cases, so we can discuss them simultaneously. We want to show that \begin{equation}\label{eq12}g(n)\cdot\frac{q^{n(m+1)}}{(q-1)^n}<(\frac{q^{m^2}}{\gcd(2,q-1)}\prod_{i=1}^{m}{(q^{2i}-1)})^{\frac{n}{3}}.\end{equation} Split a factor $(q+1)^{\frac{n}{3}}$ from the right-hand side of (\ref{eq12}) to \enquote{swallow} $g(n)$. It follows that (\ref{eq12}) is weaker than \begin{equation}\label{eq13}q^{3(m+1)}<\frac{q^{m^2}}{\gcd(2,q-1)}\prod_{i=2}^m{(q^{2i}-1)}\cdot(q-1)^4,\end{equation} which is easy to verify for all $(m,q)\not=(2,2)$.

\subsection{Case: \texorpdfstring{$S=\PO_{2m}^+(q),m\geq 4$}{S=POmegaPlus2m(q),m>=4} or \texorpdfstring{$S=\PO_{2m}^-(q),m\geq 4$}{S=POmegaMinus2m(q),m>=4}}\label{subsec4P7}

In both cases, we have $\mao(S)\leq\frac{q^{m+1}}{q-1}$ and $|S|=\frac{q^{m(m-1)}(q^m-1)}{\gcd(4,q^m-1)}\prod_{i=1}^{m-1}{(q^{2i}-1)}$, so we want to show that \begin{equation}\label{eq14}g(n)\cdot\frac{q^{n(m+1)}}{(q-1)^n}<(\frac{q^{m(m-1)}(q^m-1)}{\gcd(4,q^m-1)}\prod_{i=1}^{m-1}{(q^{2i}-1)})^{\frac{n}{3}},\end{equation} which can be done analogously to the previous case (Subsection \ref{subsec4P6}).

\subsection{Case: \texorpdfstring{$S$}{S} is an exceptional group of Lie type}

Guest, Morris, Praeger and Spiga \cite[Proof of Theorem 1.2]{GMPS15a} derived upper bounds on $\mao(S)$ for such $S$, based on the information on largest element orders of exceptional Lie type groups of odd characteristic from \cite[Table A.7]{KS09a}, the upper bounds on largest element orders for those of even characteristic from \cite[Table 5]{GMPS15a}, and information on outer automorphism group orders of such groups from \cite[Table 5, p.~xvi]{CCNPW85a}. Denoting their upper bound by $o(S)$, one can, in almost all cases, prove the sufficient inequality \begin{equation}\label{eq16}g(n)\cdot o(S)^n<|S|^{\frac{n}{3}}\end{equation} with arguments similar to those used in the nonexceptional cases. There are two particular subcases where we use a sharper upper bound on $\mao(S)$, based on reading off the precise value of $\meo(S)$ (not just an upper bound on it) and of $|\Out(S)|$ from \cite{CCNPW85a} and setting $o(S):=\meo(S)\cdot|\Out(S)|$. These two cases are $S=\leftidx{^{3}}D_4(2)$ (with $o(S)=18\cdot 3=54$) and $S=\leftidx{^{2}}F_4(2)'$ (with $o(S)=16\cdot 2=32$).

\section{Acknowledgements}

The author would like to thank the following people: Michael Giudici for raising his awareness to the paper \cite{GPS15a}, Igor Shparlinski and Alina Ostafe for a valuable discussion stimulating this improved version of a previous paper by the author, and Peter Hellekalek for his support and helpful comments.


\begin{thebibliography}{1}

\bibitem{Bor15a}
A.~Bors, Classification of finite group automorphisms with a large cycle, preprint, arXiv:1410.2284.

\bibitem{CCNPW85a}
J.H.~Conway, R.T.~Curtis, S.P.~Norton, R.A.~Parker and R.A.~Wilson, {\em Atlas of finite groups}, Clarendon Press, Oxford, 1985 (reprinted 2013).

\bibitem{GAP4}
The GAP~Group, \emph{GAP -- Groups, Algorithms, and Programming, Version 4.7.5}, 2014, \url{http://www.gap-system.org}.

\bibitem{GPS15a}
M.~Giudici, C.E.~Praeger and P.~Spiga, Finite primitive permutation groups and regular cycles of their elements, {\em J. Algebra} \textbf{421}:27--55, 2015.

\bibitem{GMPS15a}
S.~Guest, J.~Morris, C.E.~Praeger and P.~Spiga, On the maximum orders of elements of finite almost simple groups and primitive permutation groups, {\em Trans. Amer. Math. Soc.}, to appear, arXiv:1301.5166 [math.GR].

\bibitem{Hor74a}
M.V.~Horo\v{s}evski\u{\i}, On automorphisms of finite groups, {\em Math. USSR Sb.} \textbf{22}(4):584--594, 1974.

\bibitem{KS09a}
W.M.~Kantor and {\'A}.~Seress, Large element orders and the characteristic of Lie-type simple groups, {\em J.~Algebra}, \textbf{322}(3):802--832, 2009.

\bibitem{Mas84a}
J.-P.~Massias, Majoration explicite de l'ordre maximum d'un {\'e}l{\'e}ment du groupe sym{\'e}trique, {\em Ann. Fac. Sci. Toulouse Math. (5)} \textbf{6}(3--4):269--281, 1985.

\bibitem{Rob96a}
D.J.S.~Robinson, {\em A Course in the Theory of Groups}, Springer (Graduate Texts in Mathematics, 80), New York, 2nd ed.~1996.

\bibitem{Ros75a}
J.S.~Rose, Automorphism groups of groups with trivial centre, {\em Proc. London Math. Soc. (3)} \textbf{31}(2):167--193, 1975.

\bibitem{RS62a}
J.B.~Rosser and L.~Schoenfeld, Approximate formulas for some functions of prime numbers, {\em Illinois J. Math.} \textbf{6}:64--94, 1962.

\end{thebibliography}
\end{document}